\documentclass[nohyperref]{article}

\usepackage{microtype}
\usepackage{graphicx}
\usepackage{subcaption,wrapfig}
\usepackage{booktabs} %
\usepackage{tikz}
\usetikzlibrary{spy}
\usepackage{multirow, makecell}

\usepackage{hyperref}

\usepackage[accepted]{icml2022}

\usepackage{amsmath}
\usepackage{amssymb}
\usepackage{mathtools}
\usepackage{amsthm}
\usepackage[capitalize,noabbrev]{cleveref}

\theoremstyle{plain}
\newtheorem{theorem}{Theorem}[section]
\newtheorem{proposition}[theorem]{Proposition}
\newtheorem{lemma}[theorem]{Lemma}

\theoremstyle{definition}
\newtheorem{definition}[theorem]{Definition}

\theoremstyle{remark}
\newtheorem{remark}[theorem]{Remark}

\newcommand{\norm}[1]{\left|\left|{#1}\right|\right|}
\newcommand{\prox}{\operatorname{Prox}}

\newcommand{\id}{\operatorname{Id}}

\renewcommand{\Im}{\operatorname{Im}}
\newcommand{\Fix}{\operatorname{Fix}}
\newcommand{\noise}{\xi}
\newcommand{\dime}{n}

\DeclareMathOperator*{\argmin}{arg\,min}

\newcommand{\reglambda }{\lambda}
\usepackage[textsize=tiny]{todonotes}
\usepackage{amsfonts}

\icmltitlerunning{Proximal Denoiser for Convergent Plug-and-Play Optimization with Nonconvex Regularization}

\begin{document}

\twocolumn[
\icmltitle{Proximal Denoiser for Convergent Plug-and-Play Optimization with Nonconvex Regularization}

\begin{icmlauthorlist}
\icmlauthor{Samuel Hurault}{yyy}
\icmlauthor{Arthur Leclaire}{yyy}
\icmlauthor{Nicolas Papadakis}{yyy}
\end{icmlauthorlist}

\icmlaffiliation{yyy}{Univ. Bordeaux, Bordeaux INP, CNRS, IMB, UMR 5251,F-33400 Talence, France}

\icmlcorrespondingauthor{Samuel Hurault}{samuel.hurault@math.u-bordeaux.fr}

\icmlkeywords{Machine Learning, ICML}

\vskip 0.3in
]

\printAffiliationsAndNotice{} %

\begin{abstract} %
Plug-and-Play (PnP) methods solve ill-posed inverse problems through iterative proximal algorithms by replacing a proximal operator by a denoising operation. When applied with deep neural network denoisers, these methods have shown state-of-the-art visual performance for image restoration problems. However, their theoretical convergence analysis is still incomplete. Most of the existing convergence results consider nonexpansive denoisers, which is non-realistic, or limit their analysis to strongly convex data-fidelity terms in the inverse problem to solve. Recently, it was proposed to train the denoiser as a gradient descent step on a functional parameterized by a deep neural network. Using such a denoiser guarantees the convergence of the PnP version of the Half-Quadratic-Splitting (PnP-HQS) iterative algorithm.
 In this paper, we show that this gradient denoiser can actually correspond to the proximal operator of another scalar function.
 Given this new result, we exploit the convergence theory of proximal algorithms in the nonconvex setting to obtain convergence results for PnP-PGD (Proximal Gradient Descent) and PnP-ADMM (Alternating Direction Method of Multipliers).
 When built on top of a smooth gradient denoiser, we show that PnP-PGD and PnP-ADMM are convergent and target stationary points of an explicit functional.
These convergence results are confirmed with numerical experiments on deblurring, super-resolution and inpainting. \footnote{Code is available at \url{https://github.com/samuro95/Prox-PnP}.}
\end{abstract}

\section{Introduction}

    Many image restoration (IR) tasks can be addressed by solving an optimization problem
    \begin{equation}
        \label{eq:pb_general}
        x^*\in\argmin_x \reglambda  f(x)+g(x)
    \end{equation}
    where $f$ is a data-fidelity term, $g$ a regularization term and $\reglambda>0$ a parameter that controls the strength of the regularization.
    In particular, we will consider the case $f(x)=\frac1{2 \sigma^2}\norm{Ax-y}^2$ that corresponds to the linear observation model $y=A x^*+\noise$ where $y$ is the degraded image, $x^*$ is the clean image that we want to recover, $A$ is a linear operator and $\noise$ a white Gaussian noise of standard deviation~$\sigma$.
    In several applications (like deblurring or inpainting), the operator $A$ is non-invertible, thus leading to an ill-posed inverse problem.
    Including a term $g$ allows to cope with the ill-posed nature of this problem by assuming a-priori that $x$ must be regular in some sense.
    A long-standing problem consists in designing functions $g(x)$ that reflect a relevant regularity prior on $x$ while allowing for efficient numerical schemes to solve~\eqref{eq:pb_general}.

    For early image models, based e.g. on Fourier spectrum~\citep{ruderman1994statistics}, total variation~\citep{ROF}, wavelet sparsity~\citep{mallat2009sparse} or patch-based Gaussian mixtures~\citep{zoran2011learning}, the problem~\eqref{eq:pb_general} can be solved either explicitly (with linear filtering) or with provably convergent algorithms (like iterative thresholding~\citep{daubechies2004iterative}).
    More recently, tremendous progress has been made in IR by adopting deep image models that are learned from a database of clean images~\citep{lunz2018adversarial,prost2021learning,gonzalez2021solving}.
    It is still an open question to understand how  deep models can be encoded in a regularization term $g(x)$ that has sufficient properties to apply well-known optimization techniques.
    
    One fruitful way to tackle this problem is to rely on first-order proximal splitting algorithms~\citep{CombettesPesquet}.
    These algorithms operate individually on $f$ and $g$ via the proximity operator defined for a  stepsize $\tau>0$ as
    \begin{equation}\label{def:prox}
    \prox_{\tau f}(x)=\argmin_z \frac{1}{2\tau}||x-z||^2+f(z),
  \end{equation}
  Among these algorithms, the Proximal Gradient Descent~(PGD) (also called Forward-Backward Splitting) alternates between a proximal operation on $g$ and a gradient descent step on $f$, when $f$ is differentiable.
  In more general settings,  Half-Quadratic-Splitting (HQS), Alternating Direction Method of Multipliers (ADMM) or Douglas-Rachford Splitting (DRS) use proximal operations for both $f$ and $g$.
 Plug-and-Play (PnP) methods~\citep{venkatakrishnan2013plug} draw an elegant connexion between proximal methods and deep image models by replacing the proximity operator of $g$ with an image denoiser.
  State-of-the-art results for various IR problems have then been obtained with PnP methods~\citep{meinhardt2017learning,buzzard2018pnp,ahmad2020pnp,Yuan_2020_CVPR,zhang2021plug}, but with scarce theoretical guarantees. 
  Indeed, since a generic deep denoiser cannot be expressed as a proximal mapping~\citep{Moreau1965} in general, convergence results relying on properties of the proximal operator do not follow easily.
  Moreover, the regularizer $g$ is only made implicit via the denoising operation.
  Therefore, PnP algorithms do not seek the minimization of an explicit objective functional which strongly limits their interpretation and numerical control.

In this paper, we tackle these issues by shedding a new light on the gradient-step denoiser proposed in~\cite{hurault2022gradient, cohen2021has}, which writes
        \begin{equation}\label{def:gs}
        D_\sigma = \id - \nabla g_{\sigma}
        \end{equation}
where $g_\sigma : \mathbb{R}^n \to \mathbb{R}$ is a scalar function parameterized by a differentiable neural network.
\citep{hurault2022gradient} have shown that, if this denoiser is used in PnP-HQS, one can ensure convergence of the iterates towards a stationary point of an explicit functional.
 This trick works specifically for PnP-HQS and does not extend to other PnP frameworks.

    \paragraph{Contributions.} In this work, we propose to plug a denoiser trained as a proximal mapping, \emph{i.e.} an {\em implicit} gradient step. We first make use of the results from~\cite{gribonval2011should,gribonval2020characterization} on the characterization of proximity operators, to demonstrate that the gradient-step denoiser~\eqref{def:gs} can actually be the exact proximal operator of some nonconvex smooth function. With this result, we  show that PnP-PGD, PnP-ADMM and PnP-DRS algorithms are guaranteed to converge to stationary points of an explicit functional. 

     \section{Related works} \label{sec:related_works}

    PnP methods were applied successfully on several IR tasks by using off-the-shelf deep denoisers such as DnCNN~\citep{zhang2017beyond} or DRUNet~\citep{zhang2021plug}.
    These denoisers are used in various proximal algorithms such as HQS~\citep{zhang2017learning,zhang2021plug}, ADMM and DRS~\citep{romano2017little,ryu2019plug}, Proximal Gradient Descent (PGD) \citep{terris2020building}.
    While convergence of PnP methods using pseudo-linear denoisers is established~\citep{sreehari2016plug,gavaskar2020plug,nair2021fixed,chan2019performance},
    little is known about convergence with deep denoisers.

     \paragraph{Nonexpansive denoisers}  Most of existing works intend to get convergence via contractive PnP fixed-point iterations.
    For instance, convergence of various proximal algorithms can be obtained by assuming the denoiser averaged~\citep{sun2019online}, firmly nonexpansive~\citep{sun2021scalable,terris2020building} or simply nonexpansive~\citep{reehorst2018regularization,liu2021recovery}.
    As native off-the-shelf deep denoisers are generally not 1-Lipschitz, several authors proposed to train deep denoisers with constrained Lipschitz constants.
    For example,~\citet{ryu2019plug} normalize each layer individually by its spectral norm, but with such a strategy, one cannot use residual skip connections, which are widely uses in deep denoisers.
    ~\citet{terris2020building} propose instead to enrich the training loss of the denoisers with a penalization on the spectral norm of the Jacobian.
    
    However, imposing nonexpansiveness to the denoiser is very likely to hurt its denoising performance (see Appendix~\ref{app:nonexpansive_denoisers}). ~\citet{ryu2019plug} argue that it is more realistic to suppose nonexpansiveness of the residual.
    This nevertheless comes at the cost of imposing strong convexity on the data term~$f$ in the IR problem~\eqref{eq:pb_general}, which excludes many tasks like deblurring, super-resolution and inpainting.
Finally, let us recall that all the strategies based on nonexpansiveness have another major limitation.
These methods target fixed points of some (firmly) nonexpansive or contractive operators but do not minimize an explicit functional, thus making convergence difficult to monitor.

 \paragraph{Gradient step denoisers}   With the Regularization by Denoising framework (RED),~\citep{romano2017little} shows that under homogeneity, nonexpansiveness and Jacobian symmetry conditions,
 a denoiser can be written as a gradient descent step realized on a convex potential (as in~\eqref{def:gs}).
 However, as shown in~\cite{reehorst2018regularization}, such conditions are unrealistic for deep denoisers.
 As a consequence, it is proposed in~\cite{cohen2021has} and~\cite{hurault2022gradient} to directly train the denoiser as an explicit gradient step performed on a possibly nonconvex function.
Plugging this denoiser in the PnP-HQS algorithm, it provides a convergent scheme that targets a stationary point of an explicit function. However, this strategy does not extend to other PnP frameworks such as PnP-PGD or PnP-ADMM.

\paragraph{Proximal denoisers} As PnP methods were built by replacing a proximal operation by a denoiser, it is interesting to investigate under which conditions
a denoiser can actually be a proximal map. The theorem of Moreau~\citep{Moreau1965} states that we get a proximal operator of some  \emph{convex} potential if and only if it is nonexpansive and the sub-gradient of a
convex function. \citep{sreehari2016plug} exploit this result to show the convergence of PnP-ADMM, but the analysis is thus limited to nonexpansive (and pseudo-linear) denoisers.
The purpose of this paper is to extend such results to deep image denoisers that are not nonexpansive.
To do so, we will focus on learning a deep image denoiser as the proximal operator of a \emph{nonconvex} functional. %
Under this condition, we provide new proofs of convergence of the iterates of PnP-PGD,  PnP-ADMM and PnP-DRS, towards stationary points of explicit functions.

\section{Proximal gradient step denoiser}
\label{sec:prox_denoiser}

Recently introduced in~\cite{hurault2022gradient} and \cite{cohen2021has} the Gradient Step (GS) Denoiser~\eqref{def:gs}  writes
\begin{equation}
        \label{eq:GS_den}
        D_\sigma = \nabla h_{\sigma},
        \end{equation}
with a potential \begin{equation}
        \label{eq:pot_h}
        h_\sigma : x \to \frac12 \norm{x}^2 - g_\sigma(x),     \end{equation}
obtained from a scalar function $g_\sigma : \mathbb{R}^n \to \mathbb{R}$  parameterized by a differentiable neural network.
This denoiser $D_\sigma$ is trained to denoise images degraded with Gaussian noise of level $\sigma$. In~\cite{hurault2022gradient}, it is shown that, although constrained to be an exact conservative field~\eqref{eq:GS_den},
it can realize state-of-the-art denoising.

In this section, we make use of the recent works~\cite{gribonval2011should,gribonval2020characterization} on the characterization of proximity operators
to show that this gradient denoiser can be written as a proximal operator.
These results generalize the theorem of Moreau~\citep{Moreau1965} to denoisers that may not be nonexpansive.
Applying Theorem 1 in~\cite{gribonval2020characterization}, we first have that if ${h_\sigma}$ is convex, the GS denoiser ${D_\sigma = \nabla h_\sigma}$ is linked to the proximal operator
of some function $\phi_\sigma : \mathbb{R}^n \to \mathbb{R} \cup \{+\infty\}$: ${\forall x \in \mathbb{R}^n}$, ${D_\sigma(x) \in \prox_{\phi_\sigma}(x)}=\argmin_z \frac12||x-z||^2+\phi_\sigma(z)$.
The next proposition shows that, if the residual $\id -D_{\sigma}$ is contractive, there exists a closed-form and smooth $\phi_\sigma$ such that $D_\sigma = \prox_{\phi_\sigma}$ is single-valued.

\begin{proposition}[proof in Appendix~\ref{app:proof_prop}]
    \label{prop:GSPnP_prox} Let $\mathcal{X}$ be an open convex subset of $\mathbb{R}^n$
    and $g_\sigma : \mathcal{X} \to \mathbb{R}$ a $\mathcal{C}^{k+1}$ function with $k \geq 1$ and $\nabla g_\sigma$ $L$-Lipschitz with $L<1$. Then, for $h_\sigma : x \to \frac12 \norm{x}^2 - g_\sigma(x)$
    and ${D_\sigma := \nabla h_\sigma =  \id - \nabla g_\sigma}$,
    \begin{itemize}
        \item[(i)] $h_{\sigma}$ is $(1-L)$-strongly convex and $\forall x \in \mathcal{X}$, $J_{D_{\sigma}}(x)$ is positive definite
        \item[(ii)] $D_\sigma$ is injective, $\Im(D_\sigma)$ is open and, $\forall x \in \mathcal{X}$, ${D_\sigma(x)\hspace{-1pt}=\hspace{-1pt}\prox_{\phi_\sigma}(x)}$,  with ${\phi_\sigma : \mathcal{X} \to \mathbb{R} \cup \{+\infty\}}$ defined by
        \begin{equation}\hspace*{-1cm}
            \label{eq:phi}
            \phi_\sigma(x):=\left\{\begin{array}{ll}   g_\sigma({D_\sigma}^{-1}(x)))-\frac{1}{2} \norm{{D_\sigma}^{-1}(x)-x}^2  \\\hspace{3.15cm} \text{if }\ x \in \Im(D_\sigma),\\
             +\infty \hspace{2.5cm}\text{otherwise},\end{array}\right.
        \end{equation}

        \item[(iii)] $\forall x \in \mathcal{X}$, $\phi_\sigma(x) \geq g_\sigma(x)$ and for $x \in \Fix(D_\sigma)$, $\phi_\sigma(x) = g_\sigma(x)$.
        \item[(iv)] $\phi_\sigma$ is $\mathcal{C}^k$ on $\Im(D_\sigma)$ and $\forall x \in \Im(D_\sigma)$, ${\nabla \phi_\sigma(x) = {D_\sigma}^{-1}(x) - x = \nabla g_\sigma ({D_\sigma}^{-1}(x))}$
        \item[(v)]  $\nabla \phi_\sigma$ is $\frac{L}{1-L}$-Lipschitz on $\Im(D_\sigma)$
    \end{itemize}
\end{proposition}

\begin{remark}
    \label{rem:prop}
    Note that despite $\phi_\sigma$ being possibly nonconvex, ${D_\sigma = \prox_{\phi_\sigma}}$ is one-to-one. In the literature, this is referred to as prox-regularity of $\phi_\sigma$. Also note that $D_\sigma$ is possibly not nonexpansive.
\end{remark}

PnP algorithms, such as PnP-PGD or PnP-ADMM, were built by replacing one proximal operation with a denoiser~$D_\sigma$.
When used with a denoiser $D_\sigma$ satisfying the assumptions of
Proposition~\ref{prop:GSPnP_prox}, PnP algorithms become classical proximal algorithms, with the notable difference that  a nonconvex function is involved.

\section{Convergence analysis of PnP methods with proximal denoisers}
\label{sec:conv_analysis}
In this section, we study the convergence of the three algorithms PnP-PGD, PnP-ADMM and PnP-DRS,  with a plugged denoiser $D_\sigma = \prox_{\phi_\sigma}$ that corresponds to the proximal operator of a nonconvex regularization function~$\phi_\sigma$.
For that purpose, we target the objective function
\begin{equation}
    \label{eq:F_tau}
    F_{\lambda,\sigma} := \reglambda  f + \phi_\sigma.
\end{equation}
where $f$ is a (possibly nonconvex) data-fidelity term, $\lambda$ is a regularization parameter and  $\phi_{\sigma}$ is defined as in Proposition~\ref{prop:GSPnP_prox} from the function $g_\sigma$ satisfying $D_\sigma=\id-\nabla g_\sigma$.
In all this analysis, to \hbox{use Proposition}~\ref{prop:GSPnP_prox}, $g_\sigma$ is assumed $\mathcal{C}^2$ with contractive gradient ($L<1$). We also assume $f$ and $g_\sigma$ bounded from below. 
From Proposition~\ref{prop:GSPnP_prox} (iii), we get that $\phi_\sigma$ and thus $F_{\lambda,\sigma}$ are also bounded from below.

We show, under some conditions on the parameters, that all three algorithms give convergence in terms of function values, and also convergence of the iterates if the generated sequences are bounded and $F_{\lambda,\sigma}$ verifies the Kurdyka-Lojasiewicz (KL) property.
The boundedness of the generated sequences is discussed in Appendix~\ref{app:bounded_seq}.
The KL property  has been widely used to study the convergence of optimization algorithms in the nonconvex setting~\citep{attouch2010proximal, attouch2013convergence, ochs2014ipiano}.
Very large classes of functions, in particular all the real semi-algebraic functions, satisfy this property.

\subsection{Convergence of PnP-PGD}\label{ssec:PGD}

The PnP-PGD algorithm is adapted from the standard proximal gradient descent (PGD) algorithm where the proximal step is replaced by a denoising operation.
In the usual explicit gradient step $\id -\tau\lambda \nabla f$ used in PGD, we fix the stepsize $\tau = 1$.
With a regularization parameter $\lambda >0$, PnP-PGD then writes
\begin{equation}
    \label{eq:PnP-PGD}
     \left\{\begin{array}{l} z_{k+1} = x_k - \reglambda  \nabla f(x_k) \\ x_{k+1} =  D_\sigma(z_{k+1})
            \end{array}\right.
\end{equation}
At each iteration, $x_k \in \Im(D_\sigma)$, and we get from relations~\eqref{eq:phi},~\eqref{eq:F_tau} and~\eqref{eq:PnP-PGD} the value of the objective function 
\begin{equation}
    F_{\lambda, \sigma}(x_k) = \reglambda  f(x_k) +  g_\sigma(z_k) -\frac{1}{2}\norm{z_k-x_k}^2 .
\end{equation}
To study the convergence of PnP-PGD, we make use of the literature on the convergence analysis~\cite{attouch2013convergence, beck2017first,li2015accelerated} of the PGD
for minimizing the sum of two possibly nonconvex functions. We assume here the data-fidelity term $f$ differentiable with Lipschitz gradient, but not necessarily convex.

\begin{theorem}[Convergence of PnP-PGD, proof in Appendix~\ref{app:proof_thm_PGD}]
    \label{thm:PnP-PGD}
    Let $g_\sigma : \mathbb{R}^n \to \mathbb{R} \cup \{+\infty\}$ of class $\mathcal{C}^2$ with $L$-Lipschitz gradient, $L<1$, and $D_{\sigma} := \id - \nabla g_\sigma$.
    Let $\phi_{\sigma}$ defined from $g_\sigma$ and $D_\sigma$ as in~\eqref{eq:phi}.
    Let $f : \mathbb{R}^n \to \mathbb{R} \cup \{+\infty\}$ differentiable with $L_f$-Lipschitz gradient.
    Assume that $f$ and $g_\sigma$ are bounded from below. Then, for $\reglambda L_f<1$, the iterates $x_k$
    given by the iterative scheme~\eqref{eq:PnP-PGD} verify
    \begin{itemize}
        \item[(i)] $(F_{\lambda, \sigma}(x_k))$ is non-increasing and converges.
        \item[(ii)] The residual $\norm{x_{k+1}-x_k}$ converges to $0$ at rate $\min_{k \leq K} \norm{x_{k+1}-x_k}^2 = \mathcal{O}(1/K)$
        \item[(iii)] All cluster points of the sequence $(x_k)$ are stationary points of $F_{\lambda,\sigma}$.
        \item[(iv)] Additionally suppose that $f$ and $g_{\sigma}$ are respectively KL and semi-algebraic, then if the sequence $(x_k)$ is bounded,
    it converges, with finite length, to a stationary point of $F_{\lambda,\sigma}$.
    \end{itemize}
\end{theorem}

\begin{remark}
    \label{rem:relaxation}
    If  $\nabla g_\sigma$ has Lipschitz constant $L>1$,
    similar to~\cite{hurault2022gradient}, we can introduce  ${0 < \alpha < \frac{1}{L}}$ and relax the denoising operation by replacing $D_\sigma$ with ${D_\sigma^\alpha = \alpha D_\sigma + (1-\alpha)\id = \id - \alpha \nabla g_\sigma}$ in~\eqref{eq:PnP-PGD}. We can
    then define $\phi_{\sigma}^\alpha$ from $g_\sigma^\alpha =  \alpha g_\sigma$ as in~\eqref{eq:phi} and using the
    modified PnP-PGD algorithm $ x_{k+1}=D^\alpha_\sigma(x_k - \reglambda  \nabla f(x_k))$, we get the same convergence results than in Theorem~\ref{thm:PnP-PGD}.
 \end{remark}

Notice that in Theorem~\ref{thm:PnP-PGD}, the usual condition on the stepsize becomes a condition on the regularization parameter $\reglambda{L_f}<1$. The regularization trade-off parameter
is then limited by the value of $L_f$. Even if this is not usual in optimization, we argue that this is a not a problem as the regularization strength is also regulated by the $\sigma$ parameter 
which we are free to tune manually.

\subsection{Convergence of PnP-ADMM and PnP-DRS}
The classical PnP version of ADMM, PnP-ADMM, for a given stepsize $\tau>0$, can be written as
\begin{equation}
    \label{eq:PnP-ADMM}
     \left\{\begin{array}{l} y_{k+1} = \prox_{\tau \reglambda f}(x_k) \\ 
        z_{k+1} = D_\sigma(y_{k+1}+x_{k}) 
        \\ x_{k+1} =  x_{k} + (y_{k+1}-z_{k+1})
            \end{array}\right.
\end{equation}
Following~\cite{eckstein1994some}, %
one can show that PnP-ADMM is equivalent to PnP-DRS
\begin{equation}
    \label{eq:PnP-DRS}
     \left\{\begin{array}{ll} y_{k+1} &= \prox_{\tau \reglambda  f}(x_k) \\ 
     z_{k+1} &= D_\sigma(2y_{k+1}-x_{k}) \\ 
     x_{k+1} &=  x_{k} + \beta (z_{k+1}-y_{k+1})
            \end{array}\right.
\end{equation}
with a relaxation parameter $\beta=1$.

The authors of~\cite{li2016douglas} and~\cite{themelis2020douglas} propose convergence proofs of the DRS algorithm for the minimization of the sum of two nonconvex functions, one of the two functions being differentiable with  Lipschitz gradient. We will adapt their results to obtain convergence results of PnP-DRS. By equivalence,
the convergence of PnP-ADMM follows immediately.

\subsubsection{Differentiable data-fidelity term}

We first consider $f$ convex differentiable with Lipschitz gradient.
Like for PnP-PGD, we fix the stepsize at $\tau =1$ in~\eqref{eq:PnP-DRS} and use a regularization parameter $\lambda >0$. PnP-DRS with $\beta=1$ then realizes
\begin{equation}
    \label{eq:PnP-DRS2}
     \left\{\begin{array}{l} y_{k+1} = \prox_{\reglambda  f}(x_k) \\ 
        z_{k+1} = D_\sigma(2y_{k+1}-x_{k}) = \prox_{\phi_\sigma}(2y_{k+1}-x_{k})
         \\ x_{k+1} =  x_{k} + (z_{k+1}-y_{k+1})
            \end{array}\right.
\end{equation}
As in~\cite{themelis2020douglas}, we consider for the Lyapunov function the Douglas-Rachford Envelope%
\begin{equation}
    F_{\lambda, \sigma}^{DR,1}(x) = \phi_\sigma(z) + \reglambda  f(y) + \langle y-x, y-z \rangle + \frac{1}{2}\norm{y-z}^2%
 \label{eq:DRE}
\end{equation}
where, for a given $x$, $y$ and $z$ are obtained by the two first steps of the DRS iterations~\eqref{eq:PnP-DRS2}.

Under the same assumptions as Theorem~\ref{thm:PnP-PGD}, we now show that we also have convergence of PnP-DRS and PnP-ADMM towards stationary points of $F_{\lambda, \sigma}$.

\begin{theorem}[Convergence of PnP-DRS with $f$ differentiable, proof in Appendix~\ref{app:proof_thm_DRS}]
    \label{thm:PnP-DRS}
    Let ${g_\sigma : \mathbb{R}^n \to \mathbb{R} \cup \{+\infty\}}$ of class $\mathcal{C}^2$ with $L$-Lipschitz gradient with $L<1$.
    Let ${D_{\sigma} := \id - \nabla g_\sigma}$. Let $f : \mathbb{R}^n \to \mathbb{R} \cup \{+\infty\}$ be convex and differentiable with $L_f$-Lipschitz gradient.
    Assume that $f$ and $g_\sigma$ are bounded from below.
    Then, for $ \reglambda L_f<1$, the iterates $(y_k,z_k,x_k)$
    given by the iterative scheme~\eqref{eq:PnP-DRS2} verify
    \begin{itemize}
        \item[(i)] $(F_{\lambda, \sigma}^{DR,1}(x_k))$ is nonincreasing and converges. %
        \item[(ii)] The residual $\norm{y_{k}-z_k}$ converges to $0$ at rate $\min_{k \leq K} \norm{y_k - z_k}^2 = \mathcal{O}(1/K)$
        \item[(iii)] $(y_k)$ and $(z_k)$ have the same cluster points, all of them being stationary for $F_{\lambda,\sigma}$ with the same value of $F_{\lambda,\sigma}$.
        \item[(iv)] Additionally suppose that $f$ and $g_{\sigma}$ are respectively KL and semi-algebraic, then if the sequence $(y_k,z_k,x_k)$ is bounded,
    the whole sequence converges, and $y_k$ and $z_k$ converge to the same stationary point of $F_{\lambda,\sigma}$.
    \end{itemize}

\end{theorem}

\subsubsection{Non differentiable data-fidelity term}
To cope with a possibly non-differentiable data-fidelity term, and to get rid of the restriction on the trade-off parameter $\lambda$, we inverse the denoising and proximal steps in~\eqref{eq:PnP-DRS2}:
\begin{equation}
    \label{eq:PnP-DRS3}
     \left\{\begin{array}{l} y_{k+1} = D_\sigma(x_k) \\ 
        z_{k+1} = \prox_{\reglambda  f}(2y_{k+1}-x_{k}) \\ 
        x_{k+1} =  x_{k} + (z_{k+1}-y_{k+1})
            \end{array}\right.
\end{equation}
and adapt the Douglas-Rachford Envelope accordingly
\begin{align}
    F_{\lambda, \sigma}^{DR,2}(x) &= \phi_\sigma(y) + \reglambda  f(z) + \langle y-x, y-z \rangle + \frac{1}{2}\norm{y-z}^2%
    \label{eq:DRE2}
\end{align}
where, for a given $x$, $y$ and $z$ are obtained by the two first steps of the DRS iterations~\eqref{eq:PnP-DRS3}.

Convergence of this algorithm would be ensured for $\phi_\sigma$ Lipschitz differentiable on $\mathbb{R}^n$.
However, Proposition~\ref{prop:GSPnP_prox} shows that $\phi_\sigma$ is differentiable only on $\Im(D_\sigma)$.
We now show that, if the set $\Im(D_\sigma)$ is convex and ${L<\frac12}$, we still have convergence of the iterates in~\eqref{eq:PnP-DRS3}.

\begin{theorem}[Convergence of PnP-DRS with $f$ not differentiable, proof in Appendix~\ref{app:proof_thm_DRS2}]
    \label{thm:PnP-DRS2}
    Let $g_\sigma : \mathbb{R}^n \to \mathbb{R} \cup \{+\infty\}$ of class $\mathcal{C}^2$ with $L$-Lipschitz gradient with ${L < \frac12}$.
    Let ${D_{\sigma} := \id - \nabla g_\sigma}$. Assume that $\Im(D_\sigma)$ is convex. Let $f : \mathbb{R}^n \to \mathbb{R} \cup \{+\infty\}$ be proper lower-semicontinuous.
    Assume that $f$ and $g_\sigma$ are bounded from below.
    Then, $\forall \lambda>0$, the iterates $(y_k,z_k,x_k)$
    given by the iterative scheme~\eqref{eq:PnP-DRS3} verify %
    \begin{itemize}
        \item[(i)] $(F_{\lambda, \sigma}^{DR,2}(x_k))$ is nonincreasing and converging. %
        \item[(ii)] The residual $\norm{y_{k}-z_k}$ converges to $0$ at rate $\min_{k \leq K} \norm{y_k - z_k}^2 = \mathcal{O}(1/K)$ %
        \item[(iii)] For any cluster point $(y^*,z^*,x^*)$, $y^*$ and $z^*$ coincides to a stationary point of $F_{\lambda,\sigma}$. %
        \item[(iv)] Additionally suppose that $f$ and $g_{\sigma}$ are respectively KL and semi-algebraic, then if the sequence $(y_k,z_k,x_k)$ is bounded,
    the whole sequence converges, and $y_k$ and $z_k$ converge to the same stationary point of $F_{\lambda,\sigma}$.
    \end{itemize}

\end{theorem}

\begin{remark}
Theorem~\ref{thm:PnP-DRS2} requires $\Im(D_\sigma)$ to be convex. We underline that this assumption is difficult to verify in practice.
On the other hand, Theorem~\ref{thm:PnP-DRS2}  has two advantages with respect to  PnP-PGD in Theorem~\ref{thm:PnP-PGD} and PnP-DRS in  Theorem \ref{thm:PnP-DRS}. First, 
it encompasses nonconvex and non differentiable data-fidelity terms $f$. Next, the choice of the regularization parameter $\lambda$ is no more constrained (see the discussion at the end of Section~\ref{ssec:PGD}) and automatic parameter tuning methods~\cite{wei2020tfpnp} could be used.

\end{remark}

\section{Experiments}
\label{sec:experiments}

\subsection{Proximal GS denoiser}
\label{ssec:exp_denoiser}

In this section, we learn a gradient denoiser~\eqref{eq:GS_den} ${D_\sigma=\id -\nabla g_\sigma}$ that verifies the conditions of Proposition~\ref{prop:GSPnP_prox} and thus that can be written as a proximal mapping.
We first need to choose a relevant parametrization for the function $g_\sigma$. Several options are explored in \cite{cohen2021has}. As in \cite{hurault2022gradient}, we  design the regularization function $g_\sigma$ as
\begin{equation}
    \label{eq:g_sigma}
    g_\sigma(x) = \frac{1}{2}\norm{x - N_\sigma(x)}^2,
\end{equation}
where  $N_\sigma : \mathbb{R}^n \to \mathbb{R}^n$ is a $\mathcal{C}^2$ neural network.
This choice makes the function $g_\sigma$  bounded from below, as required in Theorems~\ref{thm:PnP-PGD},~\ref{thm:PnP-DRS} and~\ref{thm:PnP-DRS2}.
Next, it allows to take benefit of existing efficient denoising architectures.
Indeed, relation~\eqref{eq:g_sigma} leads to the following expression for the denoiser %
\begin{equation}
    \label{eq:D_param}
\begin{split}
    D_\sigma (x) &=  \id - \nabla g_{\sigma}(x) \\&= N_\sigma(x)+J_{N_\sigma(x)}^\top(x-N_\sigma(x)),
\end{split}
\end{equation}
where $J_{N_\sigma(x)}$ is the Jacobian of $N_\sigma$ at point $x$. This corresponds to applying the neural network $N_\sigma$ with an additive correction that makes the denoiser a conservative field.

\paragraph{Denoising Network Architecture}

    Similar to~\citep{hurault2022gradient}, we choose to parameterize $N_{\sigma}$ with  the architecture DRUNet~\citep{zhang2021plug} (represented in Appendix~\ref{app:architecture}), a U-Net in which residual blocks are integrated.
    DRUNet takes the noise level~$\sigma$ as input, which is consistent with our formulation.
    In order to ensure continuous differentiability w.r.t. the input, we change RELU activations to Softplus, which is $\mathcal{C}^\infty$.
    We also limit the number of residual blocks to~$2$ at each scale to lower the computational burden.

    \paragraph{Training details}

    We first train the GS denoiser, in the same conditions as~\citep{hurault2022gradient}, with the $L^2$ loss
    \begin{equation}
        \label{eq:L2_loss}
        \mathcal{L}(\sigma) = \mathbb{E}_{x\sim p, \noise_\sigma \sim \mathcal{N}(0,\sigma^2)} \left[\norm{D_\sigma (x+\noise_\sigma)-x}^2 \right], 
      \end{equation}
      where $p$ is the distribution of a database of clean images.
    However, the resulting denoiser does not verify  ${\nabla g_\sigma = \id - D_{\sigma}}$ contractive \emph{i.e.} with $L<1$ Lipschitz gradient, as required by Proposition~\ref{prop:GSPnP_prox}.
    
    The link between the Lipschitz constant of~$N_\sigma$ and the one of~$\nabla g_\sigma$ being difficult to establish, following~\citep{pesquet2021learning}, we 
    enforce $L<1$ by regularizing the training loss of $D_\sigma$ with the spectral norm of the Hessian of $g_\sigma$ that reads $\nabla^2 g_\sigma=J_{( \id -D_\sigma)}$.
    More specifically, we fine-tune the previously trained GS denoiser with the following loss:
    \begin{equation}
        \label{eq:FT_loss}
        \begin{split}
        \mathcal{L}_S &(\sigma) = \mathbb{E}_{x\sim p, \noise_\sigma \sim \mathcal{N}(0,\sigma^2)} \Big[\norm{D_\sigma
        (x+\noise_\sigma)-x}^2 \\ &+ \mu \max(\norm{J_{( \id -D_\sigma)}(x+\noise_\sigma)}_S,1-\epsilon)\Big]
        \end{split}
    \end{equation}
    where $\norm{.}_S$ denotes the spectral (or operator) norm. During training, it is estimated with $50$ power iterations.
    In practice, we fine-tune during $10$ epochs, with $\sigma$ ranging in $[0,25]$, using various penalization parameter~$\mu$ and with $\epsilon = 0.1$.

	\paragraph{Denoising results}

    We evaluate the PSNR performance of the proposed Prox-DRUNet denoiser~\eqref{eq:D_param}, trained with the loss~\eqref{eq:FT_loss} with different values of $\mu$.
    In Table~\ref{tab:denoising_results}, we compare, for various noise levels~$\sigma$, our model with DRUNet equipped with the same architecture $N_\sigma$ as our prox-DRUNet (2 residual blocks and softplus activations) and trained with an $L^2$ loss as in~\eqref{eq:L2_loss}. %
    Next, we present the results obtained with GS-DRUNet \cite{hurault2022gradient}, which corresponds to the denoiser~\eqref{eq:D_param} trained %
    with the loss~\eqref{eq:L2_loss}.
    We also indicate the performance of the classical FFDNet~\citep{zhang2018ffdnet} and DnCNN~\citep{zhang2017beyond} denoisers.

    \begin{table}[h]
        \centering\footnotesize\setlength\tabcolsep{3.pt}
        \begin{tabular}{c c c c c c }
            $\sigma (./255)$ & 5 & 10 & 15 & 20 & 25 \\
            \midrule
            FFDNet & $39.95$ & $35.81$ & $33.53$ & $31.99$ & $30.84$   \\
            DnCNN & $39.80$ & $35.82$& $33.55$ & $32.02$ & $30.87$ \\
            DRUNet & $40.19$ & $36.10$ & $33.85$ & $32.34$ & $31.21$  \\
            GS-DRUNet & $40.27$ & $36.16$& $33.92$ & $32.41$ & $31.28$   \\
            \midrule
            Prox-DRUNet ($\mu = 10^{-3}$) & $40.12$ & $35.93$ & $33.60$ & $32.01$ & $30.82$    \\
            Prox-DRUNet ($\mu = 10^{-2}$) & $40.04$ & $35.86$ & $33.51$ & $31.88$ & $30.64$   \\
        \end{tabular}
        \caption{Average denoising PSNR performance %
        of our prox-denoiser and compared methods on $256\times256$ center-cropped images from the CBSD68 dataset~\citep{MartinFTM01}, for various
        noise levels $\sigma$. }
        \label{tab:denoising_results}
    \end{table}

    As can be seen in Table~\ref{tab:denoising_results}, the Lipschitz fine-tuning step inevitably worsens the denoising performance, especially for larger values of $\mu$
    and  higher noise levels. Nevertheless for  $\mu=10^{-2}$ prox-DRUNet shows comparable performance with FFDNet and DnCNN.
    The decrease in PSNR when constraining the Lipschitz constant of $I-D_{\sigma}$ is thus limited (compared to imposing nonexpansiveness of $D_{\sigma}$, see Appendix~\ref{app:nonexpansive_denoisers}).
   
    \paragraph{Lipschitz constant}
	In our experiments, the Lipschitz constant of $g_{\sigma}=\id-D_\sigma$ is not hardly constrained to satisfy $L<1$. This property is rather softly enforced by penalization with the loss function~\eqref{eq:FT_loss}.  We now investigate the potential gap between the theoretical assumption $L<1$ and the practical considerations.
    We  evaluate in Table~\ref{tab:denoising_lip} the maximum value of  $\norm{J_{( \id -D_\sigma)}(x)}_S$ while denoising noisy images from the CBSD68 testing dataset.
    Here, the value  $\norm{J_{( \id -D_\sigma)}(x)}_S$ is computed by running the power method until convergence.
     Table~\ref{tab:denoising_lip} first illustrates that GS-DRUNet ($\mu = 0$) does not check the $L<1$ Lipschitz property, especially for very noisy images.
  Next, for a large enough penalization parameter ($\mu=10^{-2}$),  prox-DRUNet satisfies the $L<1$ constraint, on the CBSD68 testing dataset and  at all studied  noise levels.

    Tables~\ref{tab:denoising_results} and~\ref{tab:denoising_lip}  exhibit a clear trade-off, controlled by $\mu$, between denoising performance and Lipschitz constant.
    Contrary to nonexpansive denoisers (see Appendix~\ref{app:nonexpansive_denoisers}), our constrained prox denoiser provides denoising of high quality while satisfying the Lipschitz constraint.

    \begin{table}[h]
        \centering\footnotesize\setlength\tabcolsep{2pt}
        \begin{tabular}{c c c c c c c c }
            $\sigma (./255)$ & 0 & 5 & 10 & 15 & 20 & 25 \\
            \midrule
            GS-DRUNet ($\mu = 0$)& $0.94$ & $1.26$ & $2.47$ & $1.96$ & $2.50$ & $3.27$ \\
            Prox-DRUNet ($\mu = 10^{-2}$) & $0.87$ & $0.92$ & $0.95$ & $0.99$ & $0.96$ & $0.96$\\
            Prox-DRUNet ($\mu = 10^{-3}$) & $0.86$ & $0.94$ & $0.97$ & $0.98$ & $0.99$ & $1.19$ \\
        \end{tabular}
        \caption{Maximal value of $\norm{J_{( \id -D_\sigma})(x)}_S$ obtained with proximal denoisers~\eqref{eq:D_param} on $256\times256$ center-cropped CBSD68 dataset, for various
        noise levels $\sigma$.}
        \label{tab:denoising_lip}
    \end{table}

    We precise here that after each PnP experiment conducted with Prox-DRUNet in the following section, we empirically verified that we still have $\norm{\nabla^2 g_\sigma(x_k)}_S < 1$ on all the iterates $x_k$ where $D_{\sigma}$ was evaluated.

\subsection{PnP restoration}
\label{ssec:exp_restoration}
In this section, we apply, with the proximal denoiser Prox-DRUNET, the PnP algorithms PnP-PGD~\eqref{eq:PnP-PGD}, PnP-DRSdiff~\eqref{eq:PnP-DRS2} (\emph{diff} specifies that this PnP-DRS is dedicated to differentiable data-fidelity terms $f$) and PnP-DRS~\eqref{eq:PnP-DRS3} on deblurring super-resolution and inpainting experiments.
We seek an estimate $x$ of a clean image $x^* \in \mathbb{R}^\dime$, \hbox{from a} degraded obser\-vation
 obtained as
     $ y = A x^* + \noise_\nu\in \mathbb{R}^m$,
with $A$ a $m\times\dime$ degradation matrix and $\noise_\nu$ a \hbox{white Gaussian} noise with zero mean and standard deviation $\nu$. With this formulation the data-fidelity term
takes the form $f(x)=\frac{1}{2\nu^2}\norm{Ax-y}^2$ and the Lipschitz constant of $\nabla f$ is $L_f = \frac{1}{\nu^2}\norm{A^T A}_S$.
Convergence of PnP-PGD \eqref{eq:PnP-PGD} and PnP-DRSdiff~\eqref{eq:PnP-DRS2} are guaranteed by Theorems~\ref{thm:PnP-PGD} and~\ref{thm:PnP-DRS} for $\reglambda L_f<1$ and, with Theorem~\ref{thm:PnP-DRS2}, PnP-DRS~\eqref{eq:PnP-DRS2} converges without condition on the value of $\lambda$. 
  
	We will use for evaluation Gaussian noise with $3$ noise levels $\nu \in\{2.55,7.65,12.75\}/255$ \emph{i.e.} $\nu \in\{0.01,0.03,0.05\}$.
    For each noise level, we propose default values for the parameters $\sigma$ and $\lambda$ that we keep for both deblurring and super-resolution.
    These values are explicitly given in Appendix~\ref{app:parameters}. Note that for both PnP-PGD and PnP-DRSdiff,
    $\lambda$ is set to its maximal possible value for theoretical convergence and $\sigma$ is adjusted to the same value for both algorithms.
    Therefore PnP-PGD and PnP-DRSdiff target a stationary point of the same objective function $F_{\lambda,\sigma}$.
    The algorithm terminates when the relative difference between consecutive values of the objective function is less than $\epsilon = 10^{-8}$ or the number of iterations exceeds $K=1000$.

    The convergence Theorem~\ref{thm:PnP-DRS2} of Prox-PnP-DRS requires $L < 1/2$. As $D_{\sigma}$ is trained to ensure $L < 1$,
    as suggested in Remark~\ref{rem:relaxation}, we relax the denoising operation replacing $D_\sigma$ by ${D_\sigma^\alpha = \alpha D_\sigma + (1-\alpha)\id = \id - \alpha \nabla g_\sigma}$, with $\alpha = 1/2$.
    To illustrate the relavance of PnP-DRS, we also provide in Appendix~\ref{app:inpainting}, inpaiting experiments which involve a non differentiable data-fidelity term.

\subsubsection{Deblurring}
\label{sssec:deblurring}

   For image deblurring, the degradation operator $A=H$ is a convolution performed with circular boundary conditions.
    The proximal operator of $f$ can be effciently calculated using the discrete Fourier transform.
    
    We demonstrate the effectiveness of our method on a large variety of blur kernels (represented in Appendix~\ref{app:kernels}) and noise levels. As in~\citep{zhang2017learning, pesquet2021learning,zhang2021plug, hurault2022gradient}, we use the 8
    real-world camera shake kernels of~\citet{levin2009understanding} as well as the $9 \times9$ uniform kernel and the $25 \times 25$ Gaussian kernel with standard deviation $1.6$
    (as in~\citep{romano2017little}).
    The algorithms are initialized with the observed blurred image and robustness to initialization is discussed in Appendix~\ref{app:init}.

    Figure~\ref{fig:deblurring1} first illustrates, on the image `starfish', that the three methods give consistent deblurring results with sharp edges.
    The convergence of the curves $\min_{0 \leq i \leq k}\norm{x_{i+1}-x_i}^2$ (l) and $F_{\lambda,\sigma}(x_k)$ (i,j,k) empirically confirm the theoretical convergence results. Additional illustrations are provided in Appendix \ref{app:add_experiments}.

    Next we numerically evaluate in Table~\ref{tab:deblurring_results} the PSNR performance of our three PnP algorithms on CBSD68. We give comparisons with the deep state-of-the-art PnP methods IRCNN~\citep{zhang2017learning} and DPIR~\citep{zhang2021plug} which both apply
    the PnP-HQS algorithm with decreasing stepsize but without convergence guarantees. We also provide comparisons  with the GS-PnP-HQS method \citep{hurault2022gradient} which
    corresponds to PnP-HQS with the denoiser GS-DRUNet from Table~\ref{tab:denoising_results}. We finally indicate the deblurring performance of
    \emph{nonexp-PnP-PGD}, the PnP-PGD algorithm applied with the denoiser \emph{nonexp-DRUNet} trained to be nonexpansive (see Appendix~\ref{app:nonexpansive_denoisers}).

    Observe that, among Prox-PnP methods, Prox-PnP-DRS gives the best results. Compared to the two other algorithms, convergence is guaranteed whatever be the value of~$\lambda$, which can thus be tuned to optimize performance. Prox-PnP-DRS then shows very competitive performance with respect to the state-of-the-art deep PnP methods IRCNN and DPIR.
    Moreover, contrary to the proposed algorithms, IRCNN and DPIR do not have any guarantee of convergence and we empirically show in Appendix~\ref{app:DPIR} that DPIR actually fails to converge in practice.  
    Note also that, Prox-PnP-PGD and Prox-PnP-DRS, which share the same regularization parameter $\lambda$ in the objective function, almost always converge towards the same local minima, with very similar convergence curves (see Figure~\ref{fig:deblurring1}). %

    \begin{figure*}[ht] \centering
    \captionsetup[subfigure]{justification=centering}
    \begin{subfigure}[b]{.17\linewidth}
            \centering
            \begin{tikzpicture}[spy using outlines={rectangle,blue,magnification=5,size=1.2cm, connect spies}]
            \node {\includegraphics[height=2.5cm]{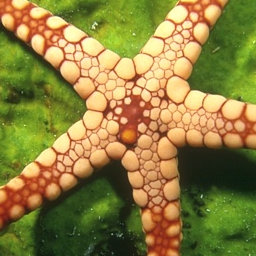}};
            \spy on (0.25,-0.4) in node [left] at  (1.25,.62);
                \node at (-0.89,-0.89) {\includegraphics[scale=1.2]{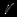}};
            \end{tikzpicture}
            \caption{Clean \\ ~}
        \end{subfigure}
    \begin{subfigure}[b]{.17\linewidth}
            \centering
            \begin{tikzpicture}[spy using outlines={rectangle,blue,magnification=5,size=1.2cm, connect spies}]
            \node {\includegraphics[height=2.5cm]{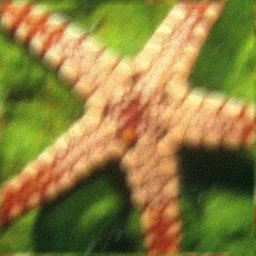}};
            \spy on (0.25,-0.4) in node [left] at  (1.25,.62);
            \end{tikzpicture}
            \caption{Observed \\ ~}
        \end{subfigure}
   \begin{subfigure}[b]{.17\linewidth}
           \centering
           \begin{tikzpicture}[spy using outlines={rectangle,blue,magnification=5,size=1.2cm, connect spies}]
           \node {\includegraphics[height=2.5cm]{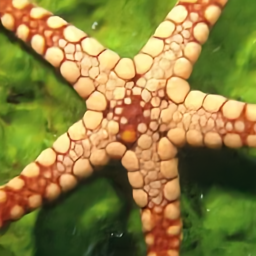}};
           \spy on (0.25,-0.4) in node [left] at  (1.25,.62);
           \end{tikzpicture}
           \caption{IRCNN \\ ($28.66$dB)}
       \end{subfigure}
   \begin{subfigure}[b]{.17\linewidth}
           \centering
           \begin{tikzpicture}[spy using outlines={rectangle,blue,magnification=5,size=1.2cm, connect spies}]
           \node {\includegraphics[height=2.5cm]{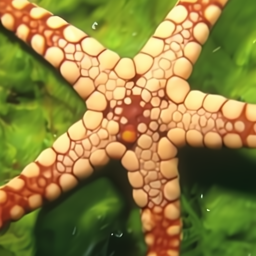}};
           \spy on (0.25,-0.4) in node [left] at  (1.25,.62);
           \end{tikzpicture}
           \caption{DPIR \\($29.76$dB)}
       \end{subfigure}
   \begin{subfigure}[b]{.17\linewidth}
           \centering
           \begin{tikzpicture}[spy using outlines={rectangle,blue,magnification=5,size=1.2cm, connect spies}]
           \node {\includegraphics[height=2.5cm]{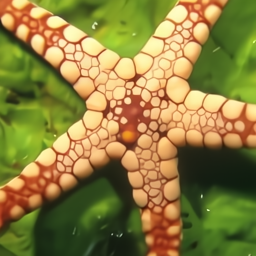}};
           \spy on (0.25,-0.4) in node [left] at (1.25,.62);
           \end{tikzpicture}
           \caption{GSPnP-HQS \\ ($29.90$dB)}
       \end{subfigure} 
    \begin{minipage}{0.52\linewidth}
    \begin{subfigure}[b]{.32\linewidth}
            \centering
            \begin{tikzpicture}[spy using outlines={rectangle,blue,magnification=5,size=1.2cm, connect spies}]
            \node {\includegraphics[width=2.5cm]{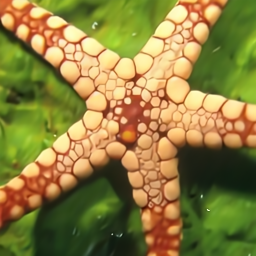}};
            \spy on (0.25,-0.4) in node [left] at (1.25,.62);
            \end{tikzpicture}
            \caption{Prox-PnP-PGD ($29.41$dB) }
        \end{subfigure}
    \begin{subfigure}[b]{.32\linewidth}
            \centering
            \begin{tikzpicture}[spy using outlines={rectangle,blue,magnification=5,size=1.2cm, connect spies}]
            \node {\includegraphics[width=2.5cm]{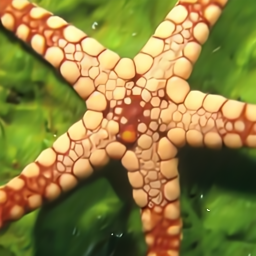}};
            \spy on (0.25,-0.4) in node [left] at (1.25,.62);
            \end{tikzpicture}
            \caption{Prox-PnP-DRSdiff ($29.41$dB)}
        \end{subfigure}
    \begin{subfigure}[b]{.32\linewidth}
            \centering
            \begin{tikzpicture}[spy using outlines={rectangle,blue,magnification=5,size=1.2cm, connect spies}]
            \node {\includegraphics[width=2.5cm]{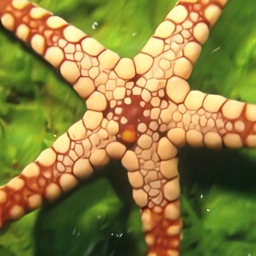}};
            \spy on (0.25,-0.4) in node [left] at (1.25,.62);
            \end{tikzpicture}
            \caption{Prox-PnP-DRS ($29.65$dB)} 
        \end{subfigure} 
   \begin{subfigure}[b]{.32\linewidth}
            \centering
            \includegraphics[width=3cm]{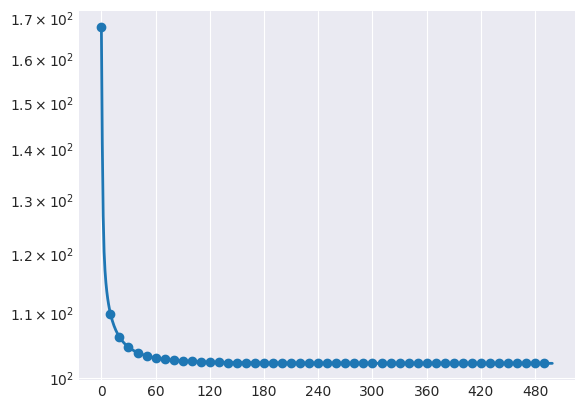}
            \caption{$F_{\lambda,\sigma}(x_k)$ \\ PnP-PGD}
        \end{subfigure}
        \begin{subfigure}[b]{.32\linewidth}
            \centering
            \includegraphics[width=3cm]{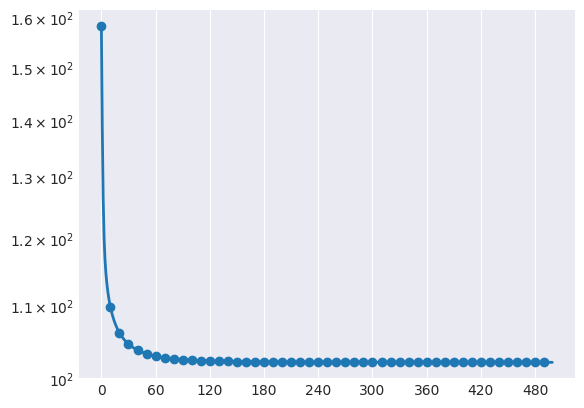}
            \caption{$F_{\lambda,\sigma}(x_k)$ \\ PnP-DRSdiff}
        \end{subfigure}
        \begin{subfigure}[b]{.32\linewidth}
            \centering
            \includegraphics[width=3cm]{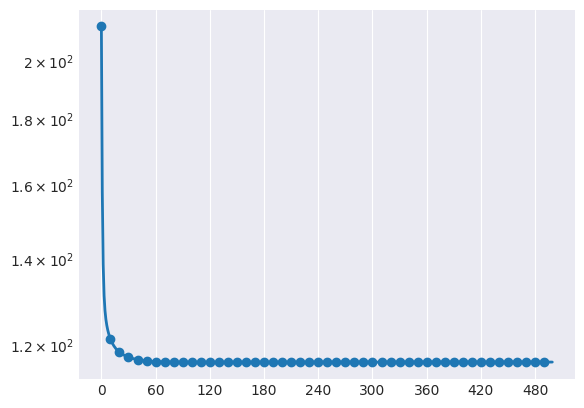}
            \caption{$F_{\lambda,\sigma}(x_k)$ \\ PnP-DRS}
        \end{subfigure}
    \end{minipage}
    \begin{minipage}{0.35\linewidth}
        \begin{subfigure}[b]{\linewidth}
            \centering
            \includegraphics[height=3.5cm]{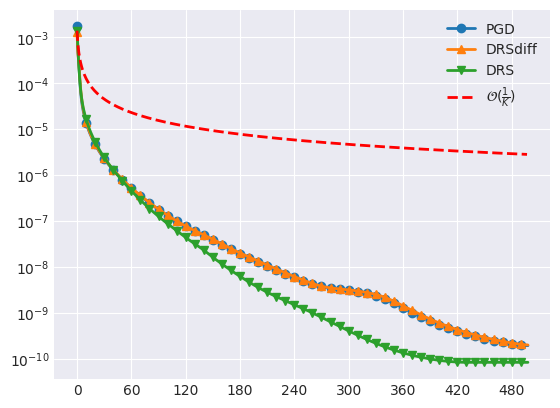}
            \caption{\footnotesize $\min_{i \leq k}\norm{x_{i+1}-x_i}^2$}
        \end{subfigure}
    \end{minipage}
    \caption{Deblurring of ``starfish" degraded with the indicated blur kernel and input noise level $\nu=0.03$.}
    \label{fig:deblurring1}
    \end{figure*}

\begin{table}[h]
    \centering\footnotesize\setlength\tabcolsep{3pt}
    \begin{tabular}{c c c c c c }
        Method & 2.55 & 7.65 & 12.75 \\
        \midrule
        IRCNN & $31.42$ & $28.01$ & $26.40$ \\
        DPIR & $31.93$ & $28.30$ & $26.82$ \\
        GS-PnP-HQS & $31.70$ & $28.28$ & $26.86$ \\
        \midrule
        Prox-PnP-PGD & $30.57$ & $27.80$ & $26.61$ \\
        Prox-PnP-DRSdiff & $30.57$ & $27.78$ & $26.61$\\
        Prox-PnP-DRS & $31.54$ & $28.07$ & $26.60$ \\
        \midrule
        nonexp-PnP-PGD & $30.25$ & $27.06$ & $25.30$ \\
    \end{tabular}
    \caption{PSNR (dB) of  deblurring methods on CBSD68. PSNR are averaged over $10$ blur kernels for various noise levels $\nu$.
    {\emph{Prox-PnP}} stands for the PnP algorithm applied with the proximal denoiser. }
    \label{tab:deblurring_results}
    \end{table}

\subsubsection{Super-resolution}
\label{sssec:super-resolution}

    \begin{figure*}[!ht] \centering
    \captionsetup[subfigure]{justification=centering}
    \begin{subfigure}[b]{.17\linewidth}
            \centering
            \begin{tikzpicture}[spy using outlines={rectangle,blue,magnification=4,size=1.2cm, connect spies}]
            \node {\includegraphics[height=2.5cm]{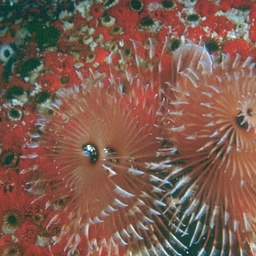}};
            \spy on (-0.6,-0.5) in node [left] at  (1.25,.62);
                \node at (-0.89,0.89) {\includegraphics[scale=0.8]{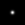}};
            \end{tikzpicture}
            \caption{Clean\\ ~}
        \end{subfigure}
    \begin{subfigure}[b]{.17\linewidth}
            \centering
            \begin{tikzpicture}[spy using outlines={rectangle,blue,magnification=4,size=.6cm, connect spies}]
            \node {\includegraphics[height=1.25cm]{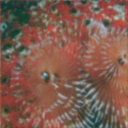}};
            \spy on (-0.3,-0.25) in node [left] at (1.25,.62);
            \end{tikzpicture}
            \caption{Observed \\ ~}
        \end{subfigure}
   \begin{subfigure}[b]{.17\linewidth}
           \centering
           \begin{tikzpicture}[spy using outlines={rectangle,blue,magnification=4,size=1.2cm, connect spies}]
           \node {\includegraphics[height=2.5cm]{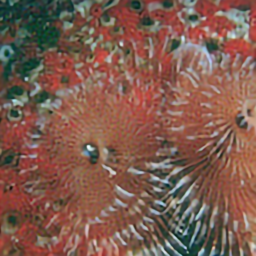}};
           \spy on  (-0.6,-0.5) in node [left] at (1.25,.62);
           \end{tikzpicture}
           \caption{IRCNN \\($27.30$dB)}
       \end{subfigure}
   \begin{subfigure}[b]{.17\linewidth}
           \centering
           \begin{tikzpicture}[spy using outlines={rectangle,blue,magnification=4,size=1.2cm, connect spies}]
           \node {\includegraphics[height=2.5cm]{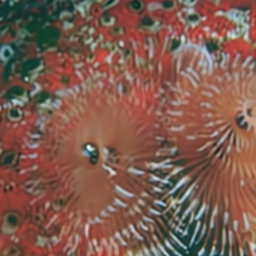}};
           \spy on  (-0.6,-0.5) in node [left] at (1.25,.62);
           \end{tikzpicture}
           \caption{DPIR \\($28.04$dB)}
       \end{subfigure}
   \begin{subfigure}[b]{.17\linewidth}
           \centering
           \begin{tikzpicture}[spy using outlines={rectangle,blue,magnification=4,size=1.2cm, connect spies}]
           \node {\includegraphics[height=2.5cm]{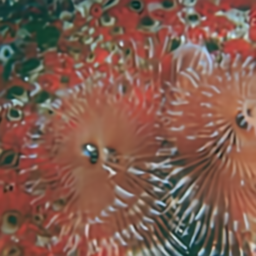}};
           \spy on  (-0.6,-0.5) in node [left] at (1.25,.62);
           \end{tikzpicture}
           \caption{GSPnP-HQS \\ ($28.20$dB)}
       \end{subfigure}
    \begin{minipage}{0.52\linewidth}
    \begin{subfigure}[b]{.32\linewidth}
            \centering
            \begin{tikzpicture}[spy using outlines={rectangle,blue,magnification=4,size=1.2cm, connect spies}]
            \node {\includegraphics[width=2.5cm]{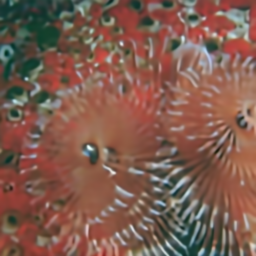}};
            \spy on (-0.6,-0.5) in node [left] at (1.25,.62);
            \end{tikzpicture}
            \caption{Prox-PnP-PGD \\ ($27.27$dB)}
        \end{subfigure}
    \begin{subfigure}[b]{.32\linewidth}
            \centering
            \begin{tikzpicture}[spy using outlines={rectangle,blue,magnification=4,size=1.2cm, connect spies}]
            \node {\includegraphics[width=2.5cm]{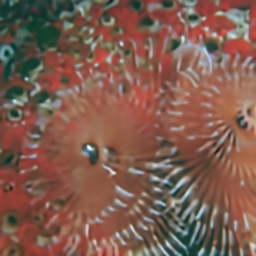}};
            \spy on (-0.6,-0.5) in node [left] at (1.25,.62);
            \end{tikzpicture}
            \caption{Prox-PnP-DRSdiff \\ ($27.27$dB)}
        \end{subfigure}
    \begin{subfigure}[b]{.32\linewidth}
            \centering
            \begin{tikzpicture}[spy using outlines={rectangle,blue,magnification=4,size=1.2cm, connect spies}]
            \node {\includegraphics[width=2.5cm]{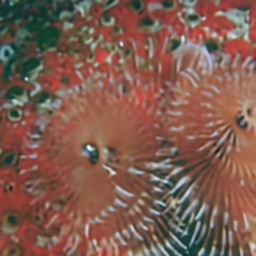}};
            \spy on (-0.6,-0.5) in node [left] at (1.25,.62);
            \end{tikzpicture}
            \caption{Prox-PnP-DRS \\ ($27.95$dB)}
        \end{subfigure} 
    \begin{subfigure}[b]{.32\linewidth}
        \centering
        \includegraphics[width=3cm]{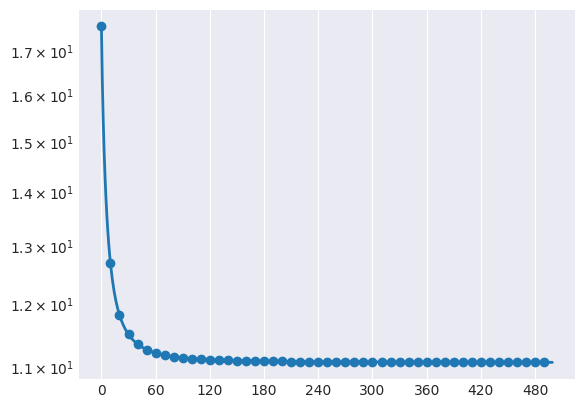}
        \caption{$F_{\lambda,\sigma}(x_k)$ \\ PnP-PGD}
    \end{subfigure}
    \begin{subfigure}[b]{.32\linewidth}
        \centering
        \includegraphics[width=3cm]{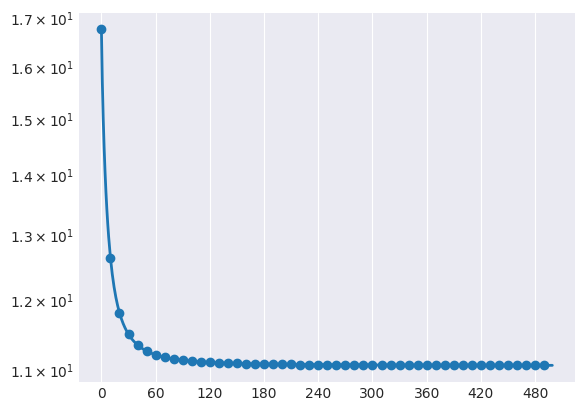}
        \caption{$F_{\lambda,\sigma}(x_k)$ \\ PnP-DRSdiff}
    \end{subfigure}
    \begin{subfigure}[b]{.32\linewidth}
        \centering
        \includegraphics[width=3cm]{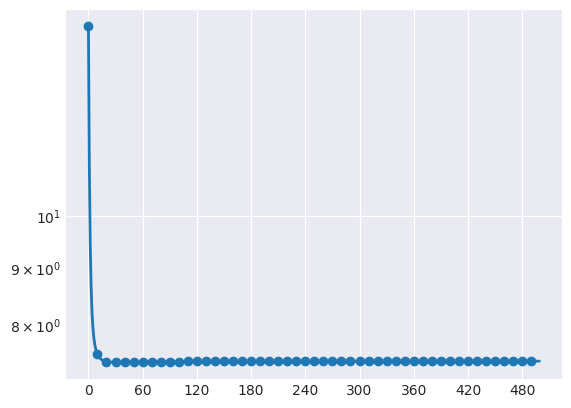}
        \caption{$F_{\lambda,\sigma}(x_k)$ \\ PnP-DRS}
    \end{subfigure}
    \end{minipage}
    \begin{minipage}{0.35\linewidth}
    \begin{subfigure}[b]{\linewidth}
        \centering
        \includegraphics[height=3.5cm]{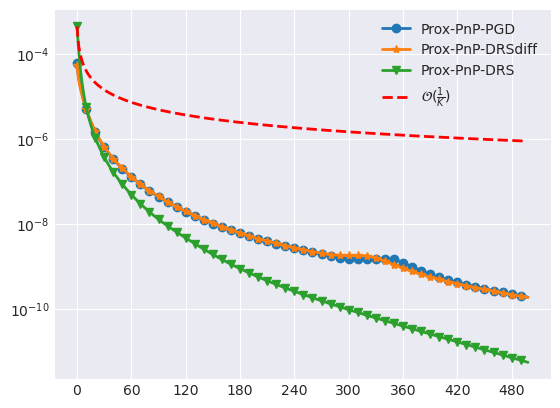}
        \caption{\footnotesize $\min_{i \leq k}\norm{x_{i+1}-x_i}^2$}
    \end{subfigure}
    \end{minipage}
    \caption{Super-resolution on an image from CBSD68 downscaled with the indicated blur kernel, scale $2$ and input noise level $\nu=0.01$.}
    \label{fig:SR1}
    \end{figure*}

\begin{table}[b]
    \centering\footnotesize\setlength\tabcolsep{3pt}%
    \begin{tabular}{c c c c c c c  }
        Method & \multicolumn{3}{c}{$s = 2$} & \multicolumn{3}{c}{$s = 3$} \\
        \cmidrule(lr){2-4} \cmidrule(lr){5-7}
         & 2.55 & 7.65 & 12.75 & 2.55 & 7.65 & 12.75 \\
        \midrule
        IRCNN & $26.97$ & $25.86$ & $25.45$ & $ 25.60$ & $ 24.72$ & $24.38$\\
        DPIR &   $27.79$ & $26.58$ & $25.83$ & $26.05$ & $25.27$ & $24.66$\\
        GS-PnP &  $27.88$ & $26.81$ & $ 26.01$ & $25.97$ & $25.35$ & $ 24.74$\\
        \midrule
        Prox-PnP-PGD & $27.44$ & $26.57$ & $25.82$ & $25.75$ & $25.20$ & $24.63$  \\
        Prox-PnP-DRSdiff & $27.44$ &  $26.58$ & $25.82$ & $25.75$ & $25.19$ & $24.63$ \\
        Prox-PnP-DRS & $27.93$ & $26.61$ & $25.79$ & $26.13$ & $25.29$ & $24.67$\\
        \midrule
        nonexp-PnP-PGD & $27.13$ & $26.20$ & $25.40$ & $23.83$ & $24.57$ & $24.01$  \\
    \end{tabular}
    \caption{PSNR (dB) of super-resolution methods on CBSD68. PSNR  averaged over $4$ blur kernels for various scales $s$ and noise levels $\nu$.}
    \label{tab:SR_results}
    \end{table}

   For single image super-resolution (SR), the
    low-resolution image $y \in \mathbb{R}^m$ is obtained from the high-resolution one $x \in \mathbb{R}^n$ via $y = SHx + \noise_\nu$ where
    $H \in \mathbb{R}^{n \times n}$ is the convolution with anti-aliasing kernel. The matrix $S$ is the standard s-fold downsampling matrix of size $m\times \dime$ and $\dime = s^2 \times m$.
   An efficient closed-form calculation of the proximal map for the data-fidelity term $f(x) = \frac{1}{2\nu^2}\norm{SHx-y}^2$ is given by~\citet{zhao2016fast}.
   
    As in~\citet{zhang2021plug}, we evaluate SR performance on 4 isotropic Gaussian blur kernels
    with different standard deviations ($0.7$, $1.2$, $1.6$ and~$2.0$) represented in Appendix~\ref{app:kernels}.
    We consider downsampled images at scale $s=2$ and $s=3$.
    On a SR example (Figure~\ref{fig:SR1}), we observe the convergence of the iterates and of the function values for all three Prox-PnP algorithms.
    The methods are  numerically compared in Table~\ref{tab:SR_results} against IRCNN, DPIR, GS-PnP and nonexp-PnP-PGD.
    Observe that, despite being trained to  guarantee convergence, the three algorithms reach the performance of DPIR and that Prox-PnP-DRS realizes, with GS-PnP, the best super-resolution performance across the variety of scales and noise levels.
    Finally note that, as for deblurring, the PnP performance is significantly reduced when  plugging a nonexpansive denoiser rather than a nonexpansive residual.%

\section{Conclusion}
In this paper we provide new convergence results for PnP schemes able to produce state-of-the-art results for image restoration.
We propose to learn a denoiser $D_\sigma$ as the proximal operator of a nonconvex regularization function.
Our proximal denoiser is not limited by nonexpansiveness and competes with unconstrained state-of-the-art denoisers.
We show that, by plugging this denoiser, PnP-PGD, PnP-ADMM and PnP-DRS algorithms are guaranteed to converge to stationary points of an explicit functional. 
PnP-PGD enables to treat data-fidelity terms for which the proximal mapping cannot be computed in closed-form.
The convergence studies for PnP-PGD/DRSdiff impose a constraint on the regularization parameter, on the contrary to PnP-DRS which thus yields better performance, at the cost of assuming convexity of $\Im(D_\sigma)$. 
Extensive quantitative experiments on ill-posed IR tasks, including data-fidelity terms that are not strongly convex, confirm the relevance of this approach, and show that it allows to precisely monitor convergence of the numerical schemes.
To better understand the applicability of theoretical results, it is important to understand the geometry of the denoiser image $\Im(D_\sigma)$, a key issue that remains to be investigated.

\section*{Acknowledgements}
This work was funded by the French ministry of research through a CDSN grant of ENS Paris-Saclay. This study has also been carried out with financial support from the French Research Agency through the PostProdLEAP and Mistic projects (ANR-19-CE23-0027-01 and ANR-19-CE40-005).

\bibliography{refs.bib}
\bibliographystyle{icml2022}
\newpage
\appendix
\onecolumn

\section{Proof of Proposition~\ref{prop:GSPnP_prox}}
\label{app:proof_prop}

Our proof uses results from~\cite{gribonval2020characterization} and borrows proofs from~\cite{gribonval2011should}. The latter shows a similar result with $D_\sigma=D_{mmse}$ the MMSE Gaussian noise denoiser.
With Tweedie's formula, $D_{mmse} = \id - \nabla g_{mmse}$ with $g_{mmse} = -\sigma^2 \log p * \mathcal{N}(0,\sigma^2 \id)$ the log-probability distribution of the noisy observation, and $g_{mmse} \in \mathcal{C}^{\infty}$.
As $g_\sigma$ does not write as $-\sigma^2 \log p * \mathcal{N}(0,\sigma^2 \id)$ for a prior $p$, the GS denoiser $D_\sigma = \id - \nabla g_\sigma$ 
does not correspond to Tweedie's formula and is not a MMSE.
The real MMSE estimator $D_\sigma^*$ cannot be computed in closed-form and $D_\sigma$ only approximates $D_\sigma^*$ via MSE minimization.
Then, the fact that $D_\sigma^*$ is a prox \cite{gribonval2011should} does not directly extend to $D_\sigma$. We show here that we can still express $D_\sigma$ as the proximal mapping of an explicit \emph{nonconvex} function. 

\begin{proof}

    \begin{itemize}

    \item[(i)] We first recall that in our setting (relations~\eqref{def:gs},~\eqref{eq:GS_den} and~\eqref{eq:pot_h}) , $D_\sigma=\id-\nabla g_\sigma$ can be written as $D_\sigma=\nabla h_\sigma$.
    We  show that if $\nabla g_\sigma = \id-\nabla h_\sigma$ is $L$-Lipschitz with $L<1$, then  $h_\sigma $ is $(1-L)$-strongly convex.
    For $x,y \in \mathcal{X}$, we have
    \begin{equation}
    \begin{split}
        \langle \nabla h_\sigma(x)- \nabla h_\sigma (y),x-y \rangle &= \norm{x-y}^2 - \langle \nabla g_\sigma(x)- \nabla g_\sigma(y),x-y \rangle  \\
        &\geq \norm{x-y}^2 -  \norm{\nabla g_\sigma(x)-\nabla g_\sigma(y)}\norm{x-y} \\
        &\geq \norm{x-y}^2 - L \norm{x-y}^2 \\
        &= (1-L) \norm{x-y}^2,
    \end{split}
    \end{equation}
    where the two inequalities respectively follow from Cauchy-Schwarz and from the Lipschitz continuity of $\nabla g_\sigma$.

    As $h_\sigma$ is $\mathcal{C}^2$ and strongly convex, it follows that $\forall x \in \mathcal{X}$, the Jacobian $J_{D_\sigma}(x) = \nabla^2 h_\sigma(x)$ is positive definite.

    \item[(ii)] Let us first show that $D_\sigma$ is injective.
    Assume by contradiction that $\exists x,x' \in \mathcal{X}$ such that $x \neq x'$ and $D_\sigma(x)=D_\sigma(x')$. Let $v = (x'-x) / \norm{x'-x}$. As $\mathcal{X}$ is convex, $\forall t \in [0,\norm{x'-x}]$, $x+tv \in \mathcal{X}$. The function $r : t \to \langle v, D_\sigma(x+tv) \rangle$
    is  $\mathcal{C}^{1}$ and satisfies $r(0) = r(\norm{x'-x})$. By Rolle's theorem, $\exists t_0 \in (0,\norm{x'-x})$ such that $r'(t_0) = \langle v, J_{D_\sigma}(x+t_0v)^Tv \rangle = 0$,
    which is impossible because $J_{D_\sigma}(x+t_0v)$ is positive definite. Therefore $D_\sigma$ is injective and we can consider its inverse $D_\sigma^{-1}$ on $\Im(D_\sigma)$.
    Also, the inverse function theorem ensures that $\Im(D_\sigma)$ is open in $\mathbb{R}^n$.

    Let $\phi_\sigma$ defined as
    \begin{equation}\label{redef:phi}
         \phi_\sigma(x) := \left\{\begin{array}{l} -\frac{1}{2} \norm{{D_\sigma}^{-1}(x)-x}^2 + g_\sigma({D_\sigma}^{-1}(x)))  \ \ \  \text{if }\ x \in \Im(D_\sigma), \\
             +\infty \ \ \ \text{otherwise}\end{array}\right.
        \end{equation}
    We now show that we also have $D_\sigma=\prox_{\phi_\sigma}$.

    For $y \in \mathcal{X}$, we thus look for the minimum of
     \begin{equation}
         \theta(x) =  \phi_\sigma(x) + \frac{1}{2}\norm{y-x}^2.
     \end{equation}

    First, the definition of $ \phi_\sigma$ ensures that $\prox_{\phi_\sigma}$ takes its values in $\Im(D_\sigma)$. We let $x = D_\sigma(u)$ for $u\in \mathcal{X}$ and define

    \begin{align}
    \begin{split}
        \Psi(u) = \theta(D_\sigma(u)) &=  \phi_\sigma(D_\sigma(u)) + \frac{1}{2}\norm{y-D_\sigma(u)}^2 \\
        &= -\frac{1}{2}\norm{u-D_\sigma(u)}^2 + g_\sigma(u) + \frac{1}{2}\norm{y-D_\sigma(u)}^2 \\
        &= -\frac{1}{2}\norm{\nabla g_\sigma(u)}^2 + g_\sigma(u) + \frac{1}{2}\norm{y-D_\sigma(u)}^2.
    \end{split}
    \end{align}
    The function $\Psi$ is $\mathcal{C}^1$ in $\mathcal{X}$. As $D_\sigma(u) = u - \nabla g_\sigma(u)$, we get $J_{D_\sigma}(u) = \id -\nabla^2 g_\sigma(u)$ (which is symmetric) and
    \begin{align}
    \begin{split}
        \nabla \Psi(u) &= -\nabla^2 g_\sigma(u). \nabla g_\sigma(u) + \nabla g_\sigma(u) - J_{D_\sigma}(u).(y-D_\sigma(u)) \\
        &= J_{D_\sigma}(u).(\nabla g_\sigma(u) - y + D_\sigma(u)) \\
        &= J_{D_\sigma}(u).(u-y).
    \end{split}
    \end{align}

    Let $r(t) :=  \Psi(y+t(u-y))$ for $t \in [0,1]$. Then,
    \begin{align}
    \begin{split}
        r'(t) &= \langle \nabla \Psi(y+t(u-y)),u-y \rangle \\
        &= t \langle J_{D_\sigma}(y+t(u-y)).(u-y),u-y \rangle, \\
    \end{split}
    \end{align}
    which is $0$ if and only if $t=0$ because $J_{D_\sigma}(y+t(u-y))$ is positive definite. Then $r$ admits a unique stationary point at $t=0$.
    This is a global minimum because $r'(t)$ has the sign of $t$.
    Therefore, $\Psi$ admits a unique stationary point at $y+0.(u-y) = y$, which is a global minimum.
    By the definition of $ \phi_\sigma$, we can deduce that $\forall y$, $x \to \phi_\sigma(x) + \frac{1}{2}\norm{y-x}^2$ admits a unique
    global minimum at $x=D_\sigma(y)$ \emph{i.e.} $D_\sigma = \prox_{ \phi_\sigma}$.

    \item[(iii)]
   For any $x \in \mathcal{X}$ we have
    \begin{equation}
        \begin{split}
            \phi_\sigma(x) &= \frac{1}{2}\norm{x-x}^2 + \phi_\sigma(x) \\
            &\geq \frac{1}{2}\norm{x-D_\sigma(x)}^2 + \phi_\sigma(D_\sigma(x)) \\
            &= g_\sigma(x) 
        \end{split}
    \end{equation}
    where the first inequality comes from the definition of the proximal operator  $D_\sigma = \prox_{\phi_\sigma}$ and the last equality is given
    by the definition~\eqref{redef:phi} of $ \phi_\sigma$.

    \item[(iv)] By the inverse function theorem, $\phi_\sigma$ is immediately $\mathcal{C}^k$ ($k \geq 1$) on $\Im(D_\sigma)$. We recall that $D_\sigma=\id-\nabla g_\sigma$. From the definition of the proximal operator of the 
    differentiable function $\phi_\sigma$, that  is single-valued and satisfies $D_\sigma=\prox_{\phi_\sigma}$ (point(ii)), we have $\forall y \in \Im(D_\sigma)$
    \begin{equation}
            D_\sigma(y) = (\id + \nabla \phi_\sigma)^{-1} (y) \quad \text{i.e.} \quad
             \nabla \phi_\sigma(y) =D_\sigma^{-1}(y) -y = \nabla g_\sigma(D_\sigma^{-1}(y)).
    \end{equation}

    \item[(v)]
We recall again that the denoiser is defined in~\eqref{eq:GS_den} as $D_\sigma=\nabla h_\sigma$, with the function  ${h_\sigma : x \to \frac{1}{2}\norm{x}^2 - g_\sigma(x)}$.
    Let $x,y \in \Im(D_\sigma)$, there exists $u,v \in \mathbb{R}^n$ such that $x=D_\sigma(u)$ and $y = D_\sigma(v)$. Hence we have
    \begin{align*}
        \norm{\nabla \phi_\sigma(x)- \nabla \phi_\sigma(y)}
        &= \norm{{D_\sigma}^{-1}(x)-{D_\sigma}^{-1}(y) - (x-y)}  \\
        &= \norm{u-D_\sigma(u) - (v-D_\sigma(v))}  \\
        &= \norm{\nabla g_\sigma(u) - \nabla g_\sigma(v)}  \\
        &\leq L \norm{u-v} \\
        &\leq \frac{L}{1-L} \norm{\nabla h_\sigma(u)- \nabla h_\sigma(v)} \\
        &= \frac{L}{1-L} \norm{D_\sigma(u)- D_\sigma(v)} \\
        &= \frac{L}{1-L} \norm{x-y},
    \end{align*}
where the first and second inequalities respectively follow from the facts that $\nabla g_\sigma$ is $L$-Lipschitz and $h_\sigma$ is $(1-L)$-strongly convex (from point (i)).

 \end{itemize}
    \end{proof}

\section{Proof of Theorem~\ref{thm:PnP-PGD}}
\label{app:proof_thm_PGD}

To prove this convergence theorem, we use the literature of convergence analysis~\cite{attouch2013convergence, beck2017first,li2015accelerated} of the PGD algorithm in the nonconvex setting.
The two first points follow exactly the same arguments as~\cite{beck2017first,hurault2022gradient}.
The third point requires an adaptation to handle the nonconvexity of $\phi_\sigma$.
Finally, the last point of the theorem directly follows from~\cite{attouch2013convergence}.

    \begin{proof}
        \item[(i)]
            We denote the proximal gradient fixed point operator ${T_{\lambda,\sigma} = \prox_{\phi_\sigma} \circ  (\id - \reglambda \nabla f)}$, the objective function
            $F_{\lambda,\sigma} = \reglambda f + \phi_\sigma$ and we introduce
            \begin{equation}\label{def:Q}
                Q(x,y) =\reglambda  f(y) +\reglambda  \langle x-y, \nabla f(y) \rangle + \frac{1}{2}\norm{x-y}^2 + \phi_\sigma(x).
            \end{equation}
            We have
                   \begin{equation}\label{relation1}
                Q(x,x) =  F_{\lambda,\sigma}(x)%
            \end{equation}
            and
            \begin{equation}
                \begin{split}
                  \argmin_x Q(x,y) &= \argmin_x \reglambda \langle x-y, \nabla f(y) \rangle + \frac{1}{2}\norm{x-y}^2 +\phi_\sigma(x) \\
                    &= \argmin_x\phi_\sigma(x) + \frac{1}{2}\norm{x-(y-\reglambda  \nabla f(y)}^2 \\
                    &= \prox_{\phi_\sigma} \circ  \left(\id - \reglambda  \nabla f\right)(y) = T_{\lambda,\sigma}(y).
                \end{split}
            \end{equation}
            The $\argmin$ is unique by Proposition~\ref{prop:GSPnP_prox} and by definition of the $\argmin$, $x_{k+1}= T_{\lambda,\sigma}(x_k)$ implies that 
             \begin{equation}\label{relation2}
             Q(x_{k+1},x_k)\leq Q(x_{k},x_k).
                         \end{equation}
            Moreover, with $f$ being $L_f$-smooth, we have by the descent lemma, for any $t \leq \frac{1}{L_f}$ and any
            ${x,y \in \mathbb{R}^n}$,
            \begin{equation}
                \label{eq:descent}
                f(x) \leq f(y) + \langle x-y, \nabla f(y) \rangle + \frac{1}{2t}\norm{x-y}^2.
            \end{equation}
            Hence for every ${x,y \in \mathbb{R}^n}$ and taking $t=\reglambda < \frac{1}{L_f}$ in relation~\eqref{eq:descent}, we get from relation~\eqref{def:Q}
            \begin{equation}\label{relation3}
             Q(x,y) \geq F_{\lambda,\sigma}(x) .
            \end{equation}
            Therefore, combining from ~\eqref{relation1},~\eqref{relation2} and~\eqref{relation3}, we get at iteration $k$,
            \begin{equation}
                F_{\lambda,\sigma}(x_{k+1}) \leq Q(x_{k+1},x_k) \leq Q(x_k,x_k) = F_{\lambda,\sigma}(x_k) .
            \end{equation}
            The sequence $(F_{\lambda,\sigma}(x_k))$ is thus non-increasing and lower-bounded by assumption. $(F_{\lambda,\sigma}(x_k))$ thus converges to a limit $F_{\lambda,\sigma}^*$.

        \item[(ii)]
              Note that $Q(x_{k+1},x_k) \leq Q(x_k,x_k)$ in~\eqref{relation2} implies
              \begin{equation}
               \phi_{\sigma}(x_{k+1}) \leq \phi_{\sigma}(x_{k}) - \reglambda \langle x_{k+1}-x_k, \nabla f(x_k) \rangle -
                \frac{1}{2}\norm{x_{k+1}-x_k}^2 .
              \end{equation}
              Using again relation ~\eqref{eq:descent} with stepsize $t=\frac{1}{L_f}$, we get
            \begin{equation}
                \label{eq:decrease_F}
                \begin{split}
                    F_{\lambda,\sigma}(x_{k+1}) &= \reglambda  f(x_{k+1}) +  \phi_{\sigma}(x_{k+1}) \\
                    &\leq  \phi_{\sigma}(x_{k}) -  \reglambda \langle x_{k+1}-x_k, \nabla f(x_k) \rangle -
                    \frac{1}{2}\norm{x_{k+1}-x_k}^2 \\
                    &+  \reglambda  f(x_{k}) +   \reglambda \langle x_{k+1}-x_k, \nabla f(x_k) \rangle +
                    \frac{L_f}{2}\norm{x_{k+1}-x_k}^2 \\
                    &= F_{\lambda,\sigma}(x_{k}) - \frac{1}{2} (1-{L_f})\norm{x_{k+1}-x_k}^2.
                \end{split}
            \end{equation}

            Summing over $k=0,1,...,m$ gives
            \begin{equation}
                \begin{split}
                    \sum_{k=0}^m \norm{x_{k+1}-x_k}^2 &\leq \frac{2}{1-L_f} \left(F_{\lambda,\sigma}(x_0) -F_{\lambda,\sigma}(x_{m+1})\right) \\
                    &\leq \frac{2}{1-L_f} \left(F_{\lambda,\sigma}(x_0)-F_{\lambda,\sigma}^*\right) .
                \end{split}
            \end{equation}
            Therefore, $\lim_{k\to\infty} \norm{x_{k+1}-x_k} = 0$ with the convergence rate  $\gamma_k = \min_{0 \leq i \leq k}\norm{x_{i+1}-x_i}^2 \leq \frac{2}{k} \frac{F_{\lambda,\sigma}(x_0)-\lim{F_{\lambda,\sigma}(x_k)}}{1-L_f}$

         \item[(iii)]
            We begin by the two following lemmas characterizing the proximal gradient descent operator $T_{\lambda,\sigma}$.
            \begin{lemma}
                \label{lem:statio}
                With the assumptions of Theorem~\ref{thm:PnP-PGD} (in particular that $\prox_{\phi_\sigma}=(\id + \partial \phi_\sigma)^{-1}$ is single-valued), for $x^* \in \mathbb{R}^n$, $x^*$ is a fixed point
                of the proximal gradient descent operator $T_{\lambda,\sigma} = \prox_{\phi_\sigma} \circ  (\id - \reglambda  \nabla f)$, \emph{i.e.} $T_{\lambda,\sigma}(x^*) = x^*$, if and only if $x^*$
                is a stationary point of $F_{\lambda,\sigma}$, \emph{i.e.} $- \reglambda \nabla  f(x^*) \in \partial \phi_\sigma (x^*)$.
            \end{lemma}

            \begin{proof}
                By definition of the proximal operator, we have
                \begin{equation}
                    \begin{split}
                        T_{\lambda,\sigma}(x^*) = x^* &\Leftrightarrow x^* = \prox_{\phi_\sigma} \circ  (\id - \reglambda  \nabla f)( x^*) \\
                          &\Leftrightarrow x^*=(\id + \partial \phi_\sigma)^{-1}(x^*-\reglambda   \nabla f(x^*))\\
                        &\Leftrightarrow - x^*+x^* -\reglambda   \nabla f(x^*)  \in \partial \phi_\sigma (x^*) \\
                        &\Leftrightarrow -\reglambda   \nabla f(x^*) \in \partial \phi_\sigma (x^*).
                    \end{split}
                \end{equation}
            \end{proof}

            \begin{lemma}
                \label{lem:lipT}
                With the assumptions of Theorem~\ref{thm:PnP-PGD}, $T_{\lambda,\sigma}$ is $(1+ \reglambda  L_f)(1+L)$ Lipschitz-continuous.
            \end{lemma}
            \begin{proof}
              The fonction $\id - \reglambda  \nabla f$ is $1+ \reglambda  L_f$-Lipschitz and $\prox_{\phi_\sigma} = D_\sigma = \id - \nabla g_\sigma$ is $1+L$-Lipschitz.
              By composition, $T_{\lambda,\sigma}$ is $(1+ \reglambda  L_f)(1+L)$-Lipschitz.
            \end{proof}

            We can now turn to the proof of (iii). Let $x^*$ be a cluster point of $(x_k)_{k\geq 0}$. Then there exists
            a subsequence $(x_{k_j})_{j \geq 0}$ converging to $x^*$.
            We have $\forall j \geq 0$,
            \begin{equation}
                \begin{split}
                    \norm{x^*-T_{\lambda,\sigma}(x^*) } &\leq \norm{x^*-x_{k_j}} + \norm{x_{k_j}-T_{\lambda,\sigma}(x_{k_j})} +  \norm{T_{\lambda,\sigma}
                    (x_{k_j})-T_{\lambda,\sigma}(x^*) } \\
                    &\leq \left( 1+(1+\reglambda  L_f)(1+L)  \right)\norm{x^*-x_{k_j}} + \norm{x_{k_j}-T_{\lambda,\sigma}(x_{k_j})} \ \text{by Lemma
                    \ref{lem:lipT}.}
                \end{split}
            \end{equation}
            Using (ii), the right-hand side of the inequality tends to $0$ as $j \to \infty$. Thus ${\norm{x^*-T_{\lambda,\sigma}(x^*)}=0}$ and $x^* = T_{\lambda,\sigma}(x^*)$, which by Lemma \ref{lem:statio} means that $x^*$ is a
            stationary point of $F_{\lambda,\sigma}$.

        \item[(iv)] The convergence of the iterates follows Theorem 5.1 from~\cite{attouch2013convergence}. To apply this result, we only need to ensure that $\phi_{\sigma}$ verifies the Kurdyka-Lojasiewicz (KL) inequality on $\Im(D_\sigma)$, since the stationary points $x^*$ of $F_{\lambda,\sigma}$ are fixed points of $T_{\lambda,\sigma}$ and belong to $\Im(D_\sigma)$. 
        We supposed $g_{\sigma}$ to be a semi-algebraic mapping.
        Semi-algebraic mappings are KL. As mentioned in~\cite{attouch2013convergence}, by the Tarski–Seidenberg theorem, the composition and inverse of
        semi-algebraic mappings are semi-algebraic mappings. Therefore, by the definition of $\phi_\sigma$, we  get that $\phi_\sigma$ is a semi-algebraic and thus KL on $\Im(D_\sigma)$.

         \end{proof}

    \section{Proofs of PnP-DRS convergence theorems}

    Theorems~\ref{thm:PnP-DRS} and~\ref{thm:PnP-DRS2} come from the convergence results presented in Theorems 4.1, 4.3 and 4.4 from~\cite{themelis2020douglas}.
    In their paper,  the following minimization problem is considered
    \begin{equation}
        \label{eq:themelis_objectif}
        \argmin_x F(x)=f_1(x) + f_2(x)
    \end{equation}
    where $f_1$ is assumed $\mathcal{C}^1$ on $\mathbb{R}^n$, $M_{f_1}$ semiconvex, with gradient $L_{f_1}$-Lipschitz.
    
    We first define the notion of \emph{semiconvexity} (also refered as \emph{weak convexity} in the literature).
    \begin{definition}[Semiconvex functions]
        $f$ proper lower semicontinuous is said to be $M$-semiconvex with $M \in \mathbb{R}$ if $x \to f(x) +  \frac{M}{2}\norm{x}^2$ is convex.
    \end{definition}

    In this setting, the authors of \cite{themelis2020douglas} show the convergence of DRS iterations
    \begin{equation}
        \label{eq:themelis_DRS}
         \left\{\begin{array}{l} y_{k+1} = \prox_{\tau f_1}(x_k) \\ z_{k+1} = \prox_{\tau f_2}(2y_{k+1}-x_{k}) \\ x_{k+1} =  x_{k} + (z_{k+1}-y_{k+1})
                \end{array}\right.
    \end{equation}
   by using the Douglas–Rachford envelope,
    \begin{equation}
    \label{eq:themelis_DRE}
    F^{DR}_{\lambda,\sigma}(x) = f_1(y) +  f_2(z) + \frac{1}{\tau} \langle y-x, y-z \rangle + \frac{1}{2\tau} \norm{y-z}^2,
    \end{equation}
where $y$ and $z$ are obtained from $x$ by the two first relations of~\eqref{eq:themelis_DRS}.

    \begin{theorem}[Convergence results of DRS from~\cite{themelis2020douglas}] \label{thm:reference_DRS}
    Let $f_1 : \mathbb{R}^n \to \mathbb{R} \cup \{+\infty\}$ and $f_2 : \mathbb{R}^n \to \mathbb{R} \cup \{+\infty\}$ be proper, lower semicontinuous functions.
    Suppose that $f_1$ is $M_{f_1}$-semiconvex and $\mathcal{C}^1$ on $\mathbb{R}^n$, with gradient $L_{f_1}$-Lipschitz.
   Then, for a stepsize
   \begin{equation}\label{thm:stepsize}
   \tau < \min \left( \frac{1}{2M_{f_1}},\frac{1}{L_{f_1}} \right),
   \end{equation}
   we have
    \begin{itemize}
        \item[(a)] (\textbf{Descent lemma}) $F^{DR}_{\lambda,\sigma}(x_k) - F^{DR}_{\lambda,\sigma}(x_{k+1}) \geq \frac{c}{(1+\tau L_{f_1})^2} \norm{x_k - x_{k+1}}^2 $ and $F^{DR}_{\lambda,\sigma}(x_k) - F^{DR}_{\lambda,\sigma}(x_{k+1}) \geq c  \norm{y_k - y_{k+1}}^2$ for some constant $c>0$ (Theorem 4.1).
        \item[(b)] $(y_k - z_k)$ vanishes with rate $\min_{k \leq K} \norm{y_k - z_k} = \mathcal{O}(\frac{1}{\sqrt{K}})$ (Theorem 4.3, (i)).
        \item[(c)] $(y_k)$ and $(z_k)$ have the same cluster points, all of them being stationary for $F$ (from equation \eqref{eq:themelis_objectif}), and on which $F$ has the same (finite) value, being the limit of $(F^{DR}_{\lambda,\sigma}(x_k))$ (Theorem 4.3, (ii))
        \item[(d)] If the sequence $(y_k,z_k,x_k)$ is bounded, and $f_1$ and $f_2$ are KL, then the sequences $(y_k)$ and $(z_k)$ converge to (the same) stationary point of $F$ (Theorem 4.4).
    \end{itemize}
    \end{theorem}

    \subsection{Proof of Theorem~\ref{thm:PnP-DRS}}
    \label{app:proof_thm_DRS}

    If $f$ is differentiable, we can directly apply Theorem~\ref{thm:reference_DRS} with $f_1 = \lambda f$ and $f_2=\phi_{\sigma}$.
    As $f$ is convex, the condition~\eqref{thm:stepsize} on the stepsize becomes $\lambda \tau < \frac{1}{L_{f}}$ \emph{i.e.} $\lambda L_f < 1$ for $\tau=1$.   

    \subsection{Proof of Theorem~\ref{thm:PnP-DRS2}}
    \label{app:proof_thm_DRS2}

We seek to apply Theorem~\ref{thm:reference_DRS} with the data-fidelity term  $f_2 = \reglambda  f$ and the regularization $f_1=\phi_\sigma$. With this setup, the DRS updates~\eqref{eq:themelis_DRS} correspond to the iterations~\eqref{eq:PnP-DRS2} and
the Douglas–Rachford envelope leads to  equation~\eqref{eq:DRE}. We have to adapt Theorem~\ref{thm:reference_DRS}, since the function $\phi_\sigma$ we consider is only smooth  on the subspace  $\Im(D_\sigma)$ and not on the whole space~$\mathbb{R}^n$ as in~\cite{themelis2020douglas}. This is the purpose of the four points (i-iv) of Theorem~\ref{thm:PnP-DRS} that correspond to  the four points (a-d) of Theorem~\ref{thm:reference_DRS} and are shown below.
\begin{proof}
\item[(i)] The proof of the descent lemma Theorem~\ref{thm:reference_DRS} (a) uses successively the following properties:
\begin{enumerate}
    \item The gradient descent lemma of a smooth and semiconvex function.
    \item The characterisation of the proximal mapping of differentiable functions : for $f$ differentiable $x = \prox_f(y)$ iif $y = x + \nabla f(x)$
    \item The strong monotonicity of the proximal operator of a smooth function.
\end{enumerate}

We consider $\Im(D_\sigma)$ convex. We first show that $\phi_\sigma$ is semiconvex on $\Im(D_\sigma)$.

$D_\sigma$ is $(L+1)$-Lipschitz. We directly apply Proposition 2 from~\cite{gribonval2020characterization} to get that there is $\tilde \phi_\sigma$ such that $\forall x \in \mathbb{R}^n, D_\sigma(x) \in \prox_{\tilde \phi_\sigma}(x)$ and 
$x \to \tilde \phi_\sigma(x) + \frac{L}{2(L+1)}\norm{x}^2$ is convex. As $\frac{L}{L+1}<1$, $\forall x \in \mathbb{R}^n$,  $y \to \tilde \phi_\sigma(y) + \frac{1}{2}\norm{x-y}^2$ is strongly convex 
and $\prox_{\tilde \phi_\sigma}(x)$ is single valued. Therefore 
\begin{equation}
    \forall x \in  \mathbb{R}^n, D_\sigma(x) = \prox_{\tilde \phi_\sigma}(x) = \prox_{\phi_\sigma}(x).
\end{equation}
For any $x \in \Im(D_\sigma)$ with $y \in \mathcal{X}$ such that $x=D_\sigma(y)$, the optimality condition of the proximal operator gives  
\begin{equation}
    \left\{\begin{array}{l} 
        y-x \in \partial \phi_\sigma(x) \\
        y-x \in \partial \tilde \phi_\sigma(x) 
    \end{array}\right.
\end{equation}
and then $\partial \phi_\sigma(x) \cap \partial \tilde \phi_\sigma(x) \neq \emptyset$.
 Moreover, $\Im(D_\sigma)$ is convex, thus connected and thus polygonally connected.
By \cite{gribonval2020characterization} Corollary 9, $\exists K \in \mathbb{R}$ such that $ \phi_\sigma = \tilde \phi_\sigma + K$ on  $\Im(D_\sigma)$.
As $\tilde \phi_\sigma$ is $\frac{L}{L+1}$-semiconvex, we get that $\phi_\sigma$ is $\frac{L}{L+1}$-semiconvex on the convex set $\Im(D_\sigma)$. 

The three points 1. 2. and 3. are then still verified on $\Im(D_\sigma)$ and Theorem~\ref{thm:reference_DRS}(a) is still valid. From (a) we get immediately that $F_{\lambda,\sigma}(x_k)$ is nonincreasing. With our hypothesis, $F_{\lambda,\sigma}(x_k)$ is bounded from below, thus converges.

    As  given in Proposition~\ref{prop:GSPnP_prox},  $\nabla \phi_\sigma$ is $\frac{L}{1-L}$ Lipschitz on $\Im(D_\sigma)$ and $\phi_\sigma$ is $\frac{L}{L+1}$ semiconvex,  with $0<L<1$.
    As we fix the stepsize to $\tau=1$ in our PnP-DRS algorithm~\eqref{eq:PnP-DRS2}, the condition~\eqref{thm:stepsize} on the stepsize in the statement of Theorem~\ref{thm:reference_DRS} becomes
    \begin{equation}
        \begin{split}
            1 &< \min \left( \frac{L+1}{2L},\frac{1-L}{L} \right) \\
            \Leftrightarrow L &< \frac{1}{2} \\
        \end{split}
    \end{equation}

\item[(ii)] As $F_{\lambda,\sigma}$ is bounded from below, the proof of Theorem~\ref{thm:reference_DRS} (b) is still valid when  $\phi_\sigma$  is smooth only on $\Im(D_\sigma)$ and we have $\min_{k \leq K} \norm{y_k - z_k} = \mathcal{O}(\frac{1}{\sqrt{K}})$.

\item[(iii)] In our setting, the limit of $y_k$ is not guaranteed to belong to $\Im(D_\sigma)$. Hence we can not follow the proof of Theorem~\ref{thm:reference_DRS} (c).
We rather rely on the convergence results from~\cite{li2016douglas}, Theorem 1, and show that
     for any cluster point $(y^*,z^*,x^*)$, we have that $y^*=z^*$ is a stationary point of $F_{\lambda,\sigma}$

Assuming that the whole sequence $(y_k,z_k,x_k)$ has a cluster point $(y^*,z^*,x^*)$, there exists  a subsequence $(y_{k_j},z_{k_j},x_{k_j})$ converging to $(y^*,z^*,x^*)$.

First, from (ii) we have that $y^* = z^*$.
Let us now show that $y^* = D_{\sigma}(x^*)$.
For $F_{\lambda,\sigma}^{DR}$ defined in~\eqref{eq:themelis_DRE},  we have from (a) that
\begin{equation}
    F_{\lambda,\sigma}^{DR}(x_k) - F_{\lambda,\sigma}^{DR}(x_{k+1}) \geq c  \norm{y_k - y_{k+1}}^2
\end{equation}
Summing from $k=1$ to $k=N-1 \geq 1$, we get
\begin{equation}
    c \sum_{k=1}^{N-1} \norm{y_k - y_{k+1}}^2 \leq F_{\lambda,\sigma}^{DR}(x_1) - F_{\lambda,\sigma}^{DR}(x_{N}) .
\end{equation}
Taking the limit $j \to +\infty$ with $N=k_j$ and recalling that $F_{\lambda,\sigma}^{DR}$ is lower semicontinuous, we obtain
\begin{equation}
    c \sum_{k=1}^{+\infty} \norm{y_k - y_{k+1}}^2 \leq F_{\lambda,\sigma}^{DR}(x_1) - F_{\lambda,\sigma}^{DR}(x_{*}) < +\infty,
\end{equation}
where the second inequality comes from the fact that $\inf F_{\lambda,\sigma}^{DR} = \inf F_{\lambda,\sigma} > +\infty $ (Theorem 3.4 from \cite{themelis2020douglas}).

Therefore $\norm{y_k - y_{k+1}} \to 0$ when $k \to +\infty$, in particular, $\lim_{j \to +\infty} y_{k_j+1} = y^*$.
However, by continuity of $D_\sigma$, $y_{k_j+1} = D_{\sigma}(x_{k_j}) \to D_{\sigma}(x^*)$. Therefore,

\begin{equation}
y^* = D_{\sigma}(x^*).
\end{equation}

We now show that $ \lim_{j \to \infty} f(z_{k_j}) = f(z^*)$. As $z_k=\prox_{\reglambda f}(y_k)$ is defined as a minimizer, we have
    \begin{equation}
        \reglambda f(z_k) + \frac{1}{2}\norm{2y_k - x_{k-1} - z_k}^2 \leq  \reglambda f(z^*) + \frac{1}{2}\norm{2y_k - x_{k-1} - z^*}^2.
    \end{equation}
Taking the limit along the convergent subsequence gives
    \begin{equation}
        \varlimsup_{j \to \infty} f(z_{k_j}) \leq f(z^*).
    \end{equation}
From the lower semicontinuity of $f$, we also have $ \varliminf_{j \to \infty} f(z_{k_j}) \geq f(z^*)$ and therefore  $ \lim_{j \to \infty} f(z_{k_j}) = f(z^*)$.

Following~\cite{li2016douglas}, the subdifferential of a proper, nonconvex function $f$ is defined as the limiting subdifferential
\begin{equation}
    \partial f(x):= \left\{ v \in \mathbb{R}^n, \exists x_k, f(x_k) \rightarrow f(x), v_k \rightarrow v,  \varliminf_{z \to x_k} \frac{f(z)-f(x_k) - \langle v_k,z-x_k \rangle }{\norm{z-x_k}} \geq 0 \ \forall k \right\}
\end{equation}
which verifies
\begin{equation}
    \label{eq:subdiff_prop}
    \left\{ v \in \mathbb{R}^n, \exists x_k, f(x_k) \rightarrow f(x), v_k \rightarrow v, v_k \in \partial f(x_k) \right\} \subseteq \partial f(x)
\end{equation}

From the definition of the $y$ and $z$ updates in~\eqref{eq:PnP-DRS2}, we have
 \begin{equation}
 \left\{
        \begin{array}{ll}
            0 &= \nabla \phi_{\sigma}(y_{k+1}) + (y_{k+1}-x_k), \\
            0 &\in \reglambda  \partial f(z_{k+1}) + (z_{k+1}-2y_{k+1} + x_k).
        \end{array}\right.
    \end{equation}
Hence summing both equations, $\forall k \geq 1$, we get
\begin{equation}
    \nabla \phi_{\sigma}(y_{k}) + (z_{k}-y_{k}) \in -\reglambda  \partial f(z_{k}) .
\end{equation}
Passing to the limit $j \to +\infty$ in the converging subsequence and using the fact that $z_{k}-y_{k} \to 0$ and $y^* = z^*$, with the property~\eqref{eq:subdiff_prop} of the
    subdifferencial, we get
\begin{equation}
    \nabla \phi_{\sigma}(y^*) \in - \reglambda \partial f(y^*)
\end{equation}
that is to say $y^* = z^*$ is a stationary point of $F_{\lambda,\sigma}$.

\item[(iv)] Point (d) is still valid  when  $\phi_\sigma$  is smooth only on $\Im(D_\sigma)$.
\end{proof}

    \section{On the boundedness of the generates sequences}
    \label{app:bounded_seq}

    In order to obtain convergence of the iterates, in the Theorems~\ref{thm:PnP-PGD},~\ref{thm:PnP-DRS} and~\ref{thm:PnP-DRS2} the generated sequences are assumed to be bounded.
    As we detail below, we observe in our experiments that boundedness is always verified. Nevertheless, we can make sure that this property is satisfied.
    A sufficient condition for the boundedness of the iterates $(x_k)$ is the coercivity of the objective function $F_{\lambda,\sigma}$, that is, $\lim_{|x| \to \infty} F_{\lambda,\sigma}(x) = +\infty$ (because the non-increasing property gives $F_{\lambda,\sigma}(x_k) \leq F_{\lambda,\sigma}(x_0)$).
    In the case of Theorem~\ref{thm:PnP-DRS}, boundedness of $(y_k)$ and $(z_k)$ then follow respectively from the fact that $\prox_{\lambda f}$ is Lipschitz (\cite{themelis2020douglas}, Proposition 2.3) and that $y_k - z_k \to 0$.  
    In the case of Theorem~\ref{thm:PnP-DRS2}, boundedness of $(y_k)$ and $(z_k)$ also follow respectively from the fact that $D_\sigma$ is Lipschitz and that $y_k - z_k \to 0$.  

    Similar to~\cite{hurault2022gradient}, we can constrain $F_{\lambda,\sigma}$ to be coercive by adding a penalization that vanishes in the neighborhood of $[0,1]^n$ and which tends to $+\infty$ at $\infty$.
    For that, we consider the ball $B(a,r)$ with $a=(1/2, \ldots, 1/2)$ and $r = \sqrt{n}$, and the penalization
    \begin{equation}
      p_{a,r}(x) = \rho(\|x-a\|^2 - r ) \quad \text{with} \
    \rho(t) = \begin{cases}
      t^3 \ \text{if} \ t > 0 \\
      0 \ \text{otherwise}
    \end{cases} .
    \end{equation}     
    It is clear that $p_{a,r}$ vanishes on $B(a,r)$ and tends to $+\infty$ at $\infty$.
    Besides, one can %
    verify that $p_{a,r}$ is $\mathcal{C}^2$ (because $\rho$ is).
    Therefore, we can add an extra term to $g_{\sigma}$, thus leading to:
    \begin{equation}
    \begin{split}
        \hat g_{\sigma}(x) &= g_{\sigma}(x) + \gamma p_{a,r}(x)
        = \frac{1}{2} \norm{x - N_\sigma(x)}^2 + \gamma p_{a,r}(x) 
    \end{split}
    \end{equation}
    and $\gamma >0$ controls the penalization strength.

    The denoising operation becomes
    \begin{equation}
        \begin{split}
            \hat D_\sigma(x) = (\id - \nabla \hat g_\sigma)(x) = D_\sigma(x) - \gamma \nabla p_{a,r}(x) .
    \end{split}
    \end{equation}

    In practice we observe that all the iterates always remain in $B(a,r)$ and that the penalization $p_{a,r}(x)$ is never activated.

    \section{DRUNet architecture}
    \label{app:architecture}
The architecture of the DRUNet \emph{light} denoiser of  (\cite{zhang2021plug}) is given in Figure \ref{fig:architecture}.
    \begin{figure}[ht] \centering
    \includegraphics[width=\linewidth]{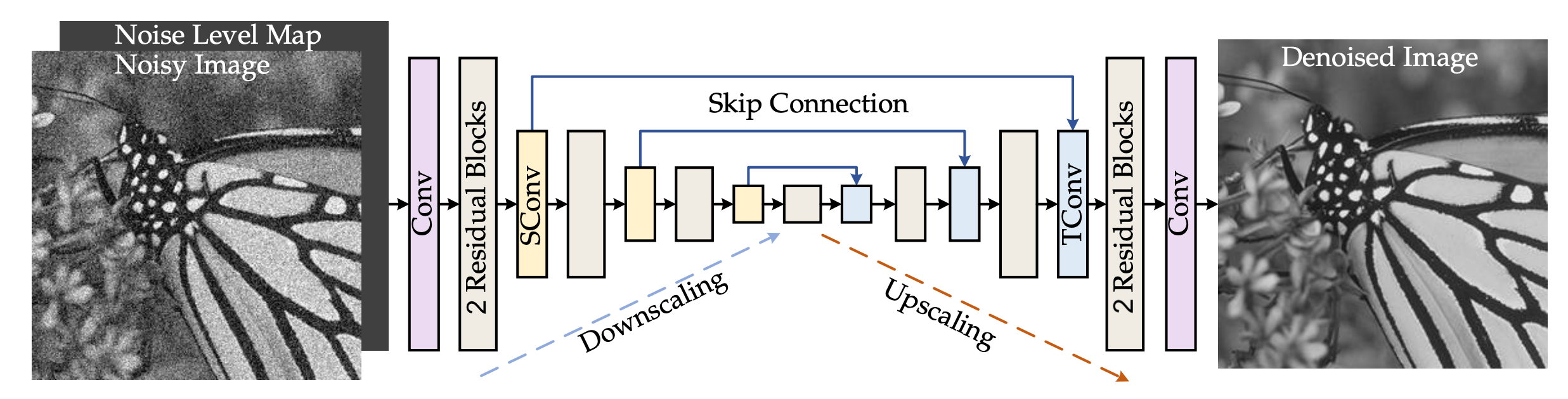}
    \caption{Architecture of the DRUNet \emph{light} denoiser (\cite{zhang2021plug}) used to parameterize $N_\sigma$.}
    \label{fig:architecture}
    \end{figure}

    \section{On the performance of nonexpansive denoisers}
    \label{app:nonexpansive_denoisers}
    
    We provide in this appendix denoising performance comparisons with a nonexpansive denoiser. To do so, we use the strategy exposed in Section~\ref{ssec:exp_denoiser} to train DRUNet
    directly to be nonexpansive. The pretrained DRUNet denoiser $N_\sigma$ is fine-tuned to denoise while being 1-Lipschitz with the following loss:
    \begin{equation}
       \mathbb{E}_{x \sim p, \xi_\sigma \sim \mathcal{N}(0,\sigma^2) }\left[ \norm{N_\sigma
        (x+\noise_\sigma)-x}^2+ \mu \max(\norm{J_{N_\sigma}(x+\noise_\sigma)}_S,1-\epsilon) \right]
    \end{equation}
    with $\epsilon = 0.1$ and different values of $\mu$.

    As in Section~\ref{ssec:exp_denoiser}, we analyse the denoising performance of the resulting denoiser nonexp-DRUNet (Table~\ref{tab:denoising_results2})
    as well as the maximal value of the spectral norm $\norm{J_{N_\sigma}(x)}_S,$ on the CBSD68 testset (Table~\ref{tab:denoising_lip2}).
    Table~\ref{tab:denoising_lip2}  illustrates that the penalization parameter has to be set to  $\mu=10^{-2}$  to make the Lipschitz constant  lower than $1$. On the other hand, as can be observed in Table~\ref{tab:denoising_results2}, such setting severely degrades denoising performances.

    From this observation, we argue that is less harmful for the denoising performance to impose nonexpansiveness on the denoiser residual $\id - D_\sigma$ than on the denoiser itself.
         \begin{table}[ht]
        \centering\footnotesize\setlength\tabcolsep{1.pt}
        \begin{tabular}{c c c c c c }
            $\sigma (./255)$ & 5 & 10 & 15 & 20 & 25 \\
            \midrule
            GS-DRUNet & $40.27$& $36.16$ & $33.92$ & $32.41$ & $31.28$   \\
            \midrule
            prox-DRUNet ($\mu = 10^{-2}$) & $40.04$ & $35.86$ & $33.51$ & $31.88$ & $30.64$   \\
            prox-DRUNet ($\mu = 10^{-3}$) & $40.12$ & $35.93$ & $33.60$ & $32.01$ & $30.82$    \\
            \midrule
            nonexp-DRUNet ($\mu = 10^{-2}$) & $34.92$ & $32.90$ & $31.42$ & $30.30$ & $29.42$    \\
            nonexp-DRUNet ($\mu = 10^{-3}$) & $39.71$ & $35.71$ & $33.50$ & $32.00$ & $30.89$   \\
        \end{tabular}
        \caption{Average PSNR denoising performance %
        of our prox-denoiser and compared methods on $256\times256$ center-cropped images from the CBSD68 dataset, for various
        noise levels $\sigma$.%
        }
        \label{tab:denoising_results2}
    \end{table}
            \begin{table}[ht]
        \centering\footnotesize\setlength\tabcolsep{2pt}
        \begin{tabular}{c c c c c c c c }
            $\sigma (./255)$ & 0 & 5 & 10 & 15 & 20 & 25 \\
            \midrule
            DRUNet  ($\mu = 0$) & $1.13$ & $1.73$ & $2.36$ & $2.67$ & $2.76$ & $3.22$\\
            nonexp-DRUNet ($\mu = 10^{-2}$) & $0.97$ & $0.98$ & $0.98$ & $0.98$ & $0.98$ & $0.98$\\
            nonexp-DRUNet ($\mu = 10^{-3}$) & $1.06$ & $1.07$ & $1.07$ & $1.10$ & $1.13$ & $1.20$\\
        \end{tabular}
        \caption{Maximal value of $\norm{J_{D_\sigma(x)}}_S$ obtained with contractive denoisers $D_\sigma=N_\sigma$ on $256\times256$ center-cropped CBSD68 dataset, for various
        noise levels $\sigma$.%
        }
        \label{tab:denoising_lip2}
    \end{table}
    \section{Experiments}
    \label{app:experiments}

    \subsection{Blur and anti-aliasing kernels}
    \label{app:kernels}

    We show  in Figures~\ref{fig:blur_kernels} and ~\ref{fig:SR_kernels} the kernels respectively used for the evaluation of
    the deblurring and SR methods.
        \begin{figure}[!ht] \centering
    \begin{tabular}{c}
    \includegraphics[width=0.08\textwidth]{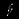}\;
    \includegraphics[width=0.08\textwidth]{images/deblurring/kernels/kernel_1.png}\;
    \includegraphics[width=0.08\textwidth]{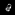}\;
    \includegraphics[width=0.08\textwidth]{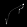}\;
    \includegraphics[width=0.08\textwidth]{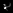}\\
    \includegraphics[width=0.08\textwidth]{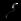}\;
    \includegraphics[width=0.08\textwidth]{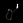}\;
    \includegraphics[width=0.08\textwidth]{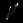}\;
    \includegraphics[width=0.08\textwidth]{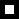}\;
    \includegraphics[width=0.08\textwidth]{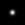}
    \end{tabular}
    \caption{The 10 blur kernels used for deblurring evaluation.\vspace{-0.3cm}
    }
    \label{fig:blur_kernels}
    \end{figure}

    \begin{figure}[!ht] \centering
    \includegraphics[width=0.08\textwidth]{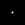}\;
    \includegraphics[width=0.08\textwidth]{images/SR/kernels/kernel_1.png}\;
    \includegraphics[width=0.08\textwidth]{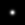}\;
    \includegraphics[width=0.08\textwidth]{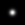}\;
    \caption{The 4 blur kernels used for super-resolution evaluation.\vspace*{-0.2cm}
    }
        \label{fig:SR_kernels}
        \end{figure}

    \subsection{Regularization parameters}
    \label{app:parameters}

    We give in Table~\ref{tab:parameters} the value of the parameters to obtain the results presented in
    Section~\ref{ssec:exp_restoration}. Note that the same parameters are taken for
    both deblurring and super-resolution, and for all the variety of considered kernels and down-sampling scales.
    They are just scaled with respect to the noise level $\nu$ in the input image.
    We keep the same parameters for PnP-PGD and PnP-DRSdiff. Both algorithms thus target a stationary points of the same objective function $F_{\lambda,\sigma}$.
    The blur kernels are normalized so that the Lipschitz constant $L_f = \frac{1}{2\nu^2} \norm{H^T H}_S$ of $\nabla f$ is equal to~$\frac{1}{2\nu^2}$.
    For PnP-PGD and PnP-DRSdiff, $\lambda$ is set as close as possible to its maximal possible value $\frac{1}{L_f}=\nu^2$  required for convergence in Theorems~\ref{thm:PnP-PGD} and~\ref{thm:PnP-DRS}.

    \begin{table}[!ht]
        \centering
        \begin{tabular}{c c c c c }
            $\nu (./255)$ &  & 2.55 & 7.65 & 12.75 \\
            \midrule
             \multirow{2}{*}{PGD \& DRSdiff}  & $\lambda/\nu^2$  & $0.99$ & $0.99$ & $0.99$\\
             &  $\sigma/\nu$&$ 0.75 $& $0.5$ & $0.5$ \\
             \midrule
            \multirow{2}{*}{DRS} & $\lambda/\nu^2$ & $5$ & $1.5$ &  $0.75$ \\
             & $\sigma/\nu$ & $2$ & $1$ & $0.5$ \\
        \end{tabular}
        \caption{Choice of parameters $(\lambda,\sigma)$ for both debluring and super-resolution experiments of  Section~\ref{ssec:exp_restoration}. These parameters are only scaled with respect to the input noise level $\nu$}
        \label{tab:parameters}
    \end{table}

 \subsection{Additional deblurring and super-resolution experiments}
    \label{app:add_experiments}

    We complement the numerical analysis with additional visual results and convergence curves.
    Figures~\ref{fig:deblurring2} and~\ref{fig:SR2} respectively give the deblurring and super-resolution results on the two images
    ''buttefly'' and ``leaves''. We also provide visual and PSNR comparison with competitive PnP methods. Note that on both experiments,
    all Prox-PnP algorithm converge towards a visually coherent output. The Prox-PnP-DRS reaches or outperforms the
    performance of the DPIR algorithm and restores sharp images.

    \begin{figure}[!ht] \centering
    \captionsetup[subfigure]{justification=centering}
    \begin{subfigure}[b]{.17\linewidth}
            \centering
            \begin{tikzpicture}[spy using outlines={rectangle,blue,magnification=4,size=1.2cm, connect spies}]
            \node {\includegraphics[height=2.5cm]{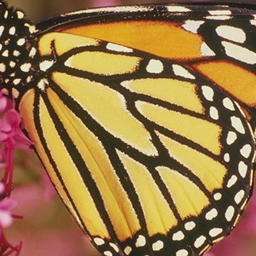}};
            \spy on (-1.1,-0) in node [left] at  (1.25,.62);
             \node at (-0.89,-0.89) {\includegraphics[scale=1.2]{images/deblurring/kernels/kernel_1.png}};
            \end{tikzpicture}
            \caption{Clean \\ ~}
        \end{subfigure}
    \begin{subfigure}[b]{.17\linewidth}
            \centering
            \begin{tikzpicture}[spy using outlines={rectangle,blue,magnification=4,size=1.2cm, connect spies}]
            \node {\includegraphics[height=2.5cm]{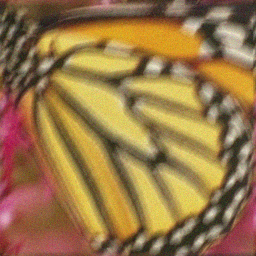}};
            \spy on (-1.1,-0) in node [left] at  (1.25,.62);
            \end{tikzpicture}
            \caption{Observed \\ ~}
        \end{subfigure}
    \begin{subfigure}[b]{.17\linewidth}
            \centering
            \begin{tikzpicture}[spy using outlines={rectangle,blue,magnification=4,size=1.2cm, connect spies}]
            \node {\includegraphics[height=2.5cm]{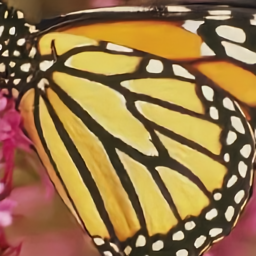}};
            \spy on (-1.1,-0) in node [left] at  (1.25,.62);
            \end{tikzpicture}
            \caption{IRCNN \\ ($28.60$dB) }
        \end{subfigure}
    \begin{subfigure}[b]{.17\linewidth}
            \centering
            \begin{tikzpicture}[spy using outlines={rectangle,blue,magnification=4,size=1.2cm, connect spies}]
            \node {\includegraphics[height=2.5cm]{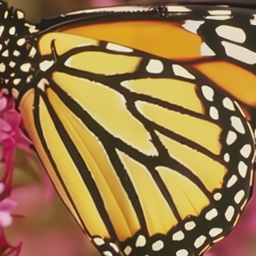}};
            \spy on (-1.1,-0) in node [left] at  (1.25,.62);
            \end{tikzpicture}
            \caption{DPIR \\($29.48$dB)}
        \end{subfigure}
    \begin{subfigure}[b]{.17\linewidth}
            \centering
            \begin{tikzpicture}[spy using outlines={rectangle,blue,magnification=4,size=1.2cm, connect spies}]
            \node {\includegraphics[height=2.5cm]{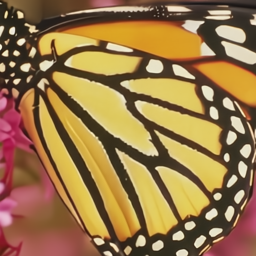}};
            \spy on (-1.1,-0) in node [left] at (1.25,.62);
            \end{tikzpicture}
            \caption{GSPnP-HQS ($29.58$dB)}
        \end{subfigure}
    \begin{minipage}{.52\linewidth}
    \begin{subfigure}[b]{.32\linewidth}
            \centering
            \begin{tikzpicture}[spy using outlines={rectangle,blue,magnification=4,size=1.2cm, connect spies}]
            \node {\includegraphics[height=2.5cm]{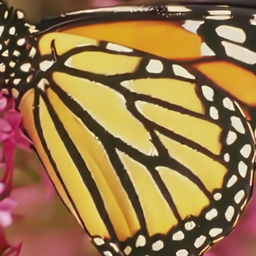}};
            \spy on (-1.1,-0) in node [left] at (1.25,.62);
            \end{tikzpicture}
            \caption{Prox-PnP-PGD \\ ($29.19$dB) }
        \end{subfigure}
    \begin{subfigure}[b]{.32\linewidth}
            \centering
            \begin{tikzpicture}[spy using outlines={rectangle,blue,magnification=4,size=1.2cm, connect spies}]
            \node {\includegraphics[height=2.5cm]{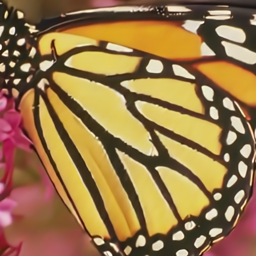}};
            \spy on  (-1.1,0) in node [left] at (1.25,.62);
            \end{tikzpicture}
            \caption{Prox-PnP-DRSdiff \\ ($29.20$dB)}
        \end{subfigure}
    \begin{subfigure}[b]{.32\linewidth}
            \centering
            \begin{tikzpicture}[spy using outlines={rectangle,blue,magnification=4,size=1.2cm, connect spies}]
            \node {\includegraphics[height=2.5cm]{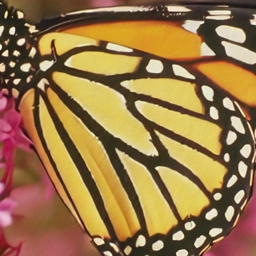}};
            \spy on  (-1.1,0)in node [left] at (1.25,.62);
            \end{tikzpicture}
            \caption{Prox-PnP-DRS \\($29.41$dB)}
        \end{subfigure} 
    \begin{subfigure}[b]{.32\linewidth}
        \centering
        \includegraphics[width=3cm]{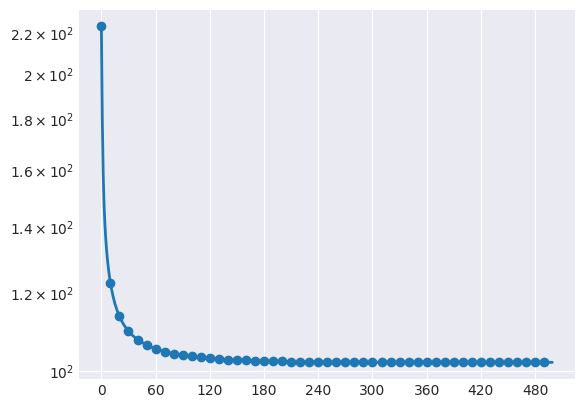}
        \caption{$F_{\lambda,\sigma}(x_k)$ \\ PnP-PGD}
    \end{subfigure}
    \begin{subfigure}[b]{.32\linewidth}
        \centering
        \includegraphics[width=3cm]{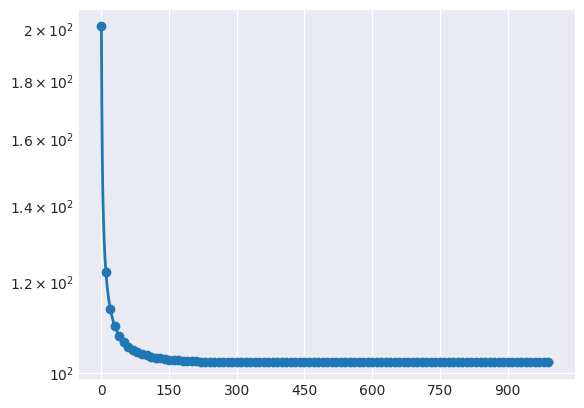}
        \caption{$F_{\lambda,\sigma}(x_k)$ \\ PnP-DRSdiff}
    \end{subfigure}
    \begin{subfigure}[b]{.32\linewidth}
        \centering
        \includegraphics[width=3cm]{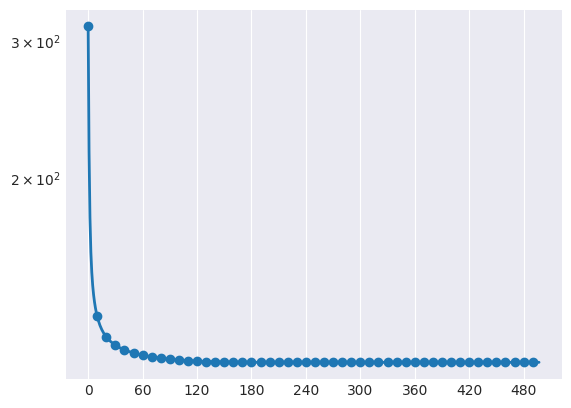}
        \caption{$F_{\lambda,\sigma}(x_k)$ \\ PnP-DRS}
    \end{subfigure}
    \end{minipage}
    \begin{minipage}{.35\linewidth}
    \begin{subfigure}[b]{\linewidth}
        \centering
        \includegraphics[height=3.5cm]{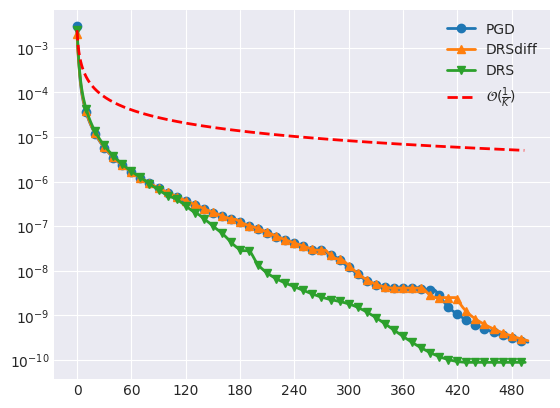}
        \caption{\footnotesize $\min_{i \leq k}\norm{x_{i+1}-x_i}^2$}
    \end{subfigure}
    \end{minipage}
    \caption{Deblurring with various methods of ``butterfly" degraded with the indicated blur kernel and input noise level $\nu=0.03$.
    }
    \label{fig:deblurring2}
    \end{figure}

    \begin{figure*}[!ht] \centering
    \captionsetup[subfigure]{justification=centering}
    \begin{subfigure}[b]{.17\linewidth}
            \centering
            \begin{tikzpicture}[spy using outlines={rectangle,blue,magnification=5,size=1.2cm, connect spies}]
            \node {\includegraphics[height=2.5cm]{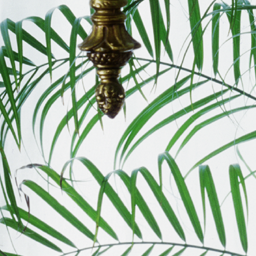}};
            \spy on (-0.3,0.18) in node [left] at  (1.25,.62);
             \node at (-0.89,-0.89) {\includegraphics[scale=0.8]{images/SR/kernels/kernel_3.png}};
            \end{tikzpicture}
            \caption{Clean \\~}
        \end{subfigure}
    \begin{subfigure}[b]{.17\linewidth}
            \centering
            \begin{tikzpicture}[spy using outlines={rectangle,blue,magnification=5,size=.6cm, connect spies}]
            \node {\includegraphics[height=1.25cm]{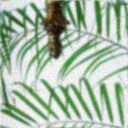}};
            \spy on (-0.15,0.09) in node [left] at  (1.5,1.25);
            \end{tikzpicture}
            \caption{Observed\\~}
        \end{subfigure}
    \begin{subfigure}[b]{.17\linewidth}
            \centering
            \begin{tikzpicture}[spy using outlines={rectangle,blue,magnification=5,size=1.2cm, connect spies}]
            \node {\includegraphics[height=2.5cm]{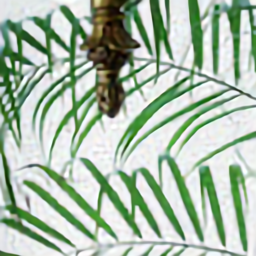}};
            \spy on (-0.3,0.18) in node [left] at (1.25,.62);
            \end{tikzpicture}
            \caption{IRCNN \\($22.82$dB)}
        \end{subfigure}
    \begin{subfigure}[b]{.17\linewidth}
            \centering
            \begin{tikzpicture}[spy using outlines={rectangle,blue,magnification=5,size=1.2cm, connect spies}]
            \node {\includegraphics[height=2.5cm]{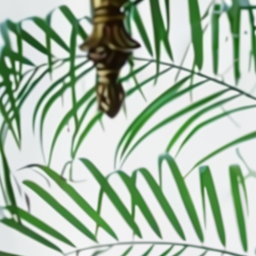}};
            \spy on (-0.3,0.18) in node [left] at (1.25,.62);
            \end{tikzpicture}
            \caption{DPIR \\($23.97$dB)}
        \end{subfigure}
    \begin{subfigure}[b]{.17\linewidth}
            \centering
            \begin{tikzpicture}[spy using outlines={rectangle,blue,magnification=5,size=1.2cm, connect spies}]
            \node {\includegraphics[height=2.5cm]{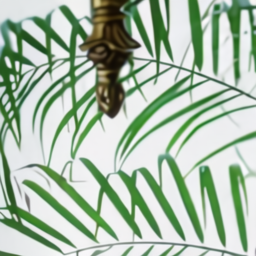}};
            \spy on (-0.3,0.18) in node [left] at (1.25,.62);
            \end{tikzpicture}
            \caption{GSPnP-HQS \\($24.81$dB)}
        \end{subfigure}
    \begin{minipage}{.52\linewidth}
    \begin{subfigure}[b]{.32\linewidth}
            \centering
            \begin{tikzpicture}[spy using outlines={rectangle,blue,magnification=5,size=1.2cm, connect spies}]
            \node {\includegraphics[height=2.5cm]{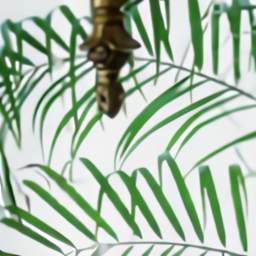}};
            \spy on (-0.3,0.18) in node [left] at (1.25,.62);
            \end{tikzpicture}
            \caption{pGSPnP-PGD ($23.96$dB) }
        \end{subfigure}
    \begin{subfigure}[b]{.32\linewidth}
            \centering
            \begin{tikzpicture}[spy using outlines={rectangle,blue,magnification=5,size=1.2cm, connect spies}]
            \node {\includegraphics[height=2.5cm]{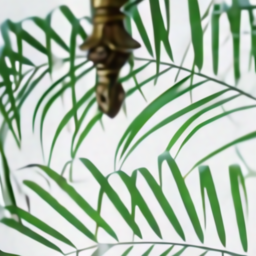}};
            \spy on (-0.3,0.18) in node [left] at (1.25,.62);
            \end{tikzpicture}
            \caption{pGSPnP-DRSdiff ($23.96$dB)}
        \end{subfigure}
    \begin{subfigure}[b]{.32\linewidth}
            \centering
            \begin{tikzpicture}[spy using outlines={rectangle,blue,magnification=5,size=1.2cm, connect spies}]
            \node {\includegraphics[height=2.5cm]{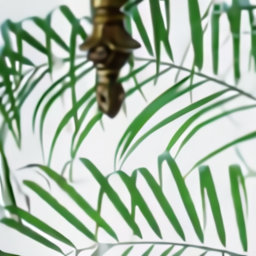}};
            \spy on (-0.3,0.18) in node [left] at (1.25,.62);
            \end{tikzpicture}
            \caption{pGSPnP-DRS ($24.36$dB)}
        \end{subfigure} \\
    \begin{subfigure}[b]{.32\linewidth}
        \centering
        \includegraphics[width=3cm]{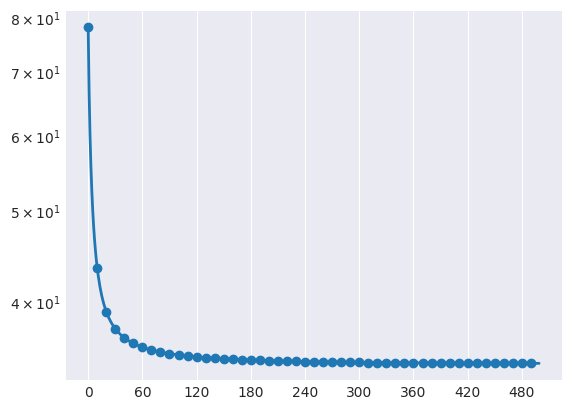}
        \caption{$F_{\lambda,\sigma}(x_k)$ \\ PnP-PGD}
    \end{subfigure}
    \begin{subfigure}[b]{.32\linewidth}
        \centering
        \includegraphics[width=3cm]{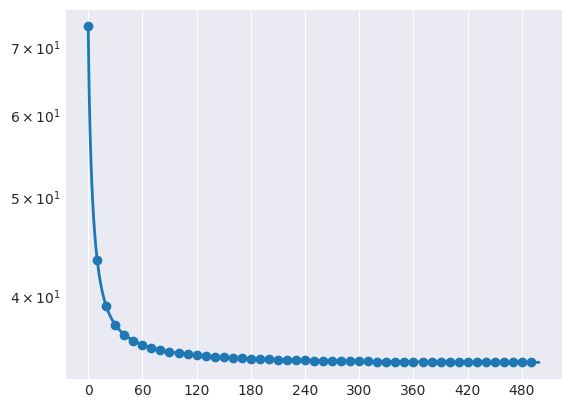}
        \caption{$F_{\lambda,\sigma}(x_k)$ \\ PnP-DRSdiff}
    \end{subfigure}
    \begin{subfigure}[b]{.32\linewidth}
        \centering
        \includegraphics[width=3cm]{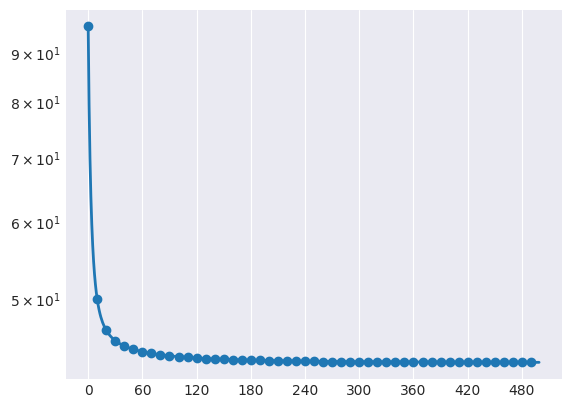}
        \caption{$F_{\lambda,\sigma}(x_k)$ \\ PnP-DRS}
    \end{subfigure}
    \end{minipage}
    \begin{minipage}{.35\linewidth}
    \begin{subfigure}[b]{\linewidth}
        \centering
        \includegraphics[height=3.5cm]{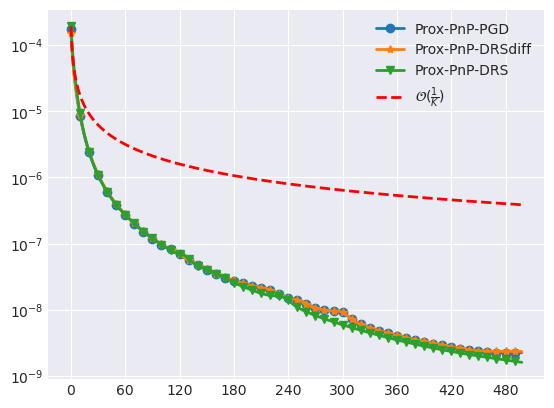}
        \caption{$\min_{i \leq k}\norm{x_{i+1}-x_i}^2$}
    \end{subfigure}
    \end{minipage}
    \caption{Super-resolution with various methods of ``leaves" degraded with the indicated blur kernel, scale $2$ and input noise level $\nu=0.03$.}
    \label{fig:SR2}
    \end{figure*}

    \clearpage

    \subsection{Robustness to initialization}
    \label{app:init}

Table~\ref{tab:init} shows the behavior of the three Prox-PnP methods %
initialized with the degraded image $y$ + random Gaussian noise. The methods are robust for small noise perturbations and converge to different local minima for stronger noise.
PnP-DRS remains consistent for stronger noise as it is run with a larger $\lambda$ value, which tends to ``convexify" the objective function.
\begin{table}[!ht]
    \centering
    \begin{tabular}{cccc}
        $\sigma_{init}$ $(./255)$ & PGD & DRSdiff & DRS \\
        \midrule
        $0$ & $29.41$ & $29.41$ & $29.65$ \\ 
        $10$& $29.41$ & $29.41$ & $29.65$ \\ 
        $20$ & $27.34$ & $25.92$ & $29.65$ \\ 
        $40$ & $12.52$ & $12.54$ & $29.63$ \\
        $60$ & $12.27$ & $12.28$ & $18.65$ \vspace{-0.2cm}
    \end{tabular}    
    \caption{\label{tab:init}\footnotesize PSNR after initialization with $x_0 = y + n$ with \hbox{$n \sim \mathcal{N}(0,\sigma^2_{init}\id)$.}}
\end{table}

\subsection{Convergence comparison with DPIR}
\label{app:DPIR}

DPIR \cite{zhang2021plug} performs well in PSNR when run on very few iterations but we show Figure~\ref{fig:DPIR} below that it actually does not converge asymptotically, contrary to our PnP methods. 
Instead of being runned on $8$ iterations, DPIR is here runned on $2000$ iterations. We observe that the both the PSNR and the norm of the residual do not converge in practice.  Such a divergence strongly limits DPIR outputs interpretation and numerical control.

\begin{figure}[ht] \centering
    \begin{subfigure}[b]{0.45\linewidth}
        \centering
        \includegraphics[scale=0.4]{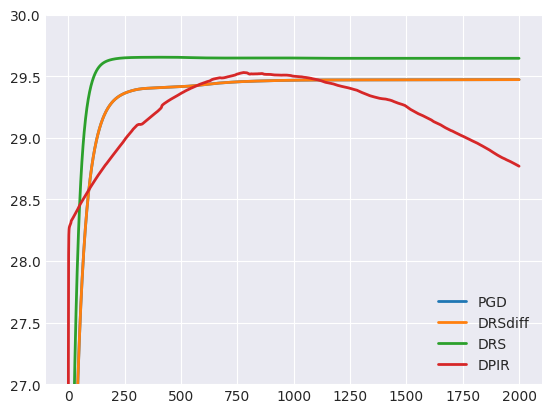}
        \caption{PSNR($x_k)$}
    \end{subfigure}
\begin{subfigure}[b]{.45\linewidth}
    \centering
       \includegraphics[scale=0.4]{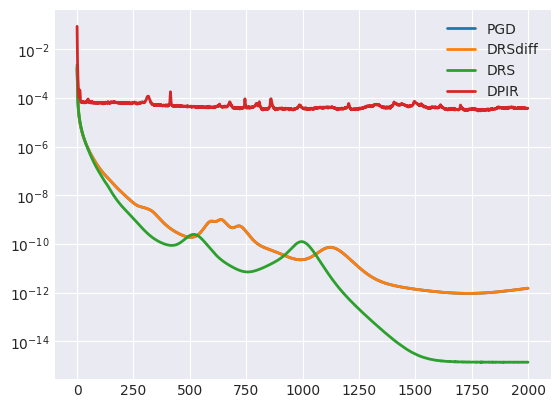}
        \caption{$\norm{x_{k+1}-x_k}^2$}
    \end{subfigure}
    \caption{Evolution, for DPIR and our PnP algorithms, on $2000$ iterations of the PSNR and of the squared norm of the residual when deblurring the blurred "starfish" image from~Figure~\ref{fig:deblurring1}.}
    \label{fig:DPIR}
\end{figure}

 \subsection{Inpainting experiment}
    \label{app:inpainting}

    We finally apply PnP-DRS to image inpainting,  with the degradation model $y = Ax$
    where $A$ is a diagonal matrix with values in $\{0,1\}$.
 In this context, the data-fidelity term is the indicator function of the set $\mathcal{A} = \{x \ | \ Ax=y\}$: $f(x) = \imath_{\mathcal{A}}(x)$ which, by definition, equals $0$ on $\mathcal{A}$ and $+\infty$ elsewhere. The convergence of PnP-DRS for such a non-differentiable data-fidelity term is  ensured by Theorem~\ref{thm:PnP-DRS2}.
    In our experiments, the diagonal of $A$ is filled with Bernoulli random variables with parameter $p=0.5$. We run our algorithm with  $\sigma = 10/255$.
   As can be observed in  Figure~\ref{fig:inpainting}, PnP-DRS restores the input images with high accuracy. Furthermore, convergence of the residual at rate $\mathcal{O}(\frac{1}{k})$ is empirically confirmed.

 In our experiments, the diagonal of $A$ is filled with Bernoulli random variables with parameter $p=0.5$.
    We run our PnP algorithm with $\lambda = 2$ and $\sigma = 15/255$.
    The algorithm is initialized with $10$ iterations of the algorithm at larger noise level $\sigma = 50/255$ andterminates when the number of iterations exceeds~$K=200$.
 \begin{figure}[ht] %
        \begin{subfigure}[b]{.24\linewidth}\centering
            \includegraphics[scale=0.35]{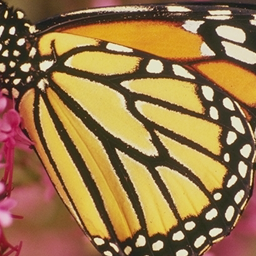}
            \caption*{Clean\\~}
        \end{subfigure}
    \begin{subfigure}[b]{.24\linewidth}\centering
           \includegraphics[scale=0.35]{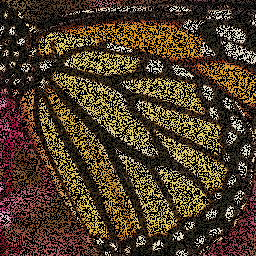}
            \caption*{Observed \\~}
        \end{subfigure}
    \begin{subfigure}[b]{.24\linewidth}\centering
            \includegraphics[scale=0.35]{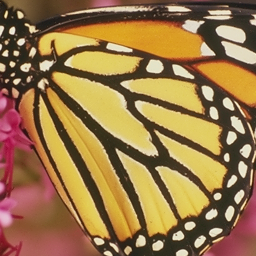}
            \caption*{Prox-PnP-DRS\\ ($31.14$dB)}
        \end{subfigure}
    \begin{subfigure}[b]{.24\linewidth}\centering
        \includegraphics[height=3.5cm]{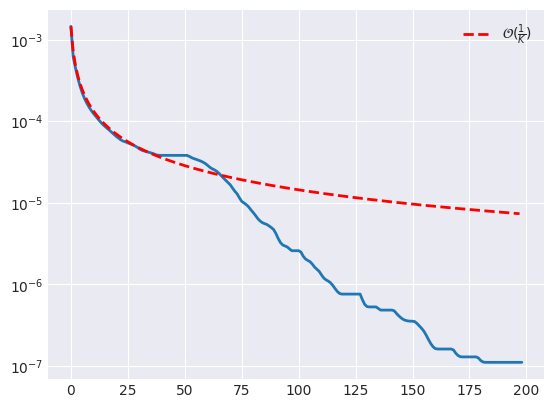}
        \caption*{$\min_{0 \leq i \leq k}\norm{x_{i+1}-x_i}^2$}
    \end{subfigure}
    
    \begin{subfigure}[b]{.24\linewidth}\centering
            \includegraphics[scale=0.35]{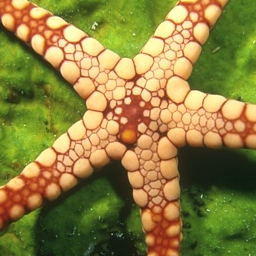}
            \caption*{Clean \\~}
        \end{subfigure}
    \begin{subfigure}[b]{.24\linewidth}\centering
           \includegraphics[scale=0.35]{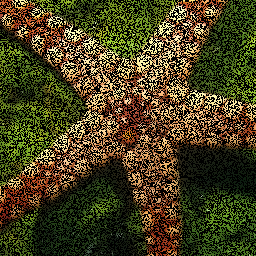}
            \caption*{Observed \\~}\centering
        \end{subfigure}
    \begin{subfigure}[b]{.24\linewidth}\centering
            \includegraphics[scale=0.35]{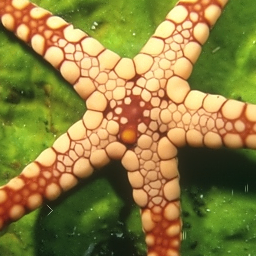}
            \caption*{Prox-PnP-DRS \\  ($33.18$dB)}
        \end{subfigure}
    \begin{subfigure}[b]{.24\linewidth}
        \centering
        \includegraphics[height=3.5cm]{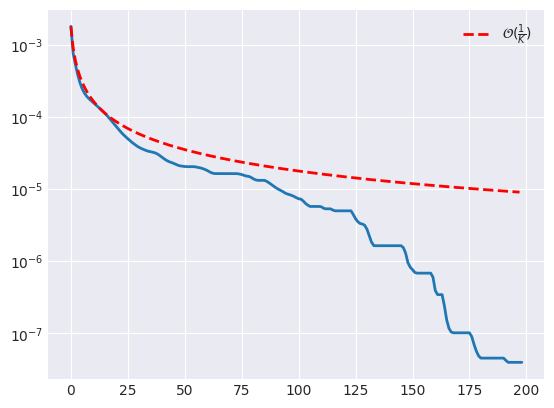}
        \caption*{$\min_{0 \leq i \leq k}\norm{x_{i+1}-x_i}^2$}
    \end{subfigure}
    \caption{Inpainting results  with pixels randomly masked with probability ${p=0.5}$.  In the last column, we show the evolution of  along the iterations.}
    \label{fig:inpainting}
    \end{figure}

\end{document}